\documentclass[10pt, a4paper, twoside]{amsart}
\usepackage{amsfonts}
\usepackage{mathrsfs}
\usepackage{amsmath}
\usepackage{amssymb}
\usepackage{fancyhdr}
\usepackage{graphicx}

\numberwithin{equation}{section}

\allowdisplaybreaks

\newcommand\HH{\mathbb{H}}

\newcommand\blambda{{\boldsymbol\lambda}}
\newcommand\bmu{{\boldsymbol\mu
}}
\newcommand\bzeta{{\boldsymbol\zeta
}}
\newcommand\bsigma{{\boldsymbol\sigma
}}

\newtheorem{theorem}{Theorem}
\newtheorem{lemma}[theorem]{Lemma}
\newtheorem{proposition}[theorem]{Proposition}

\theoremstyle{remark}
\newtheorem{remark}[theorem]{Remark}

\usepackage[T1]{fontenc}
\usepackage[utf8]{inputenc}
\usepackage[french,english]{babel}

\makeatletter

\newcommand{\abstractin}[1]{%
  \otherlanguage{#1}%
  \item[\hskip\labelsep\scshape\abstractname.]%
}
\makeatother

\begin{document}

\title{\textbf{Resolvent and Spectral Measure on Non-Trapping Asymptotically Hyperbolic Manifolds III: Global-in-Time Strichartz Estimates without Loss}\\ \textbf{(R\'{e}solvante et mesure spectrale sur des vari\'{e}t\'{e}s non-captives \`{a} asymptotes hyperboliques III: Estimations de Strichartz sans perte et globales en temps)}}



\author{Xi Chen}

\begin{abstract}

In the present paper, we investigate global-in-time Strichartz estimates without loss
on non-trapping asymptotically hyperbolic manifolds. Due to the hyperbolic nature of such
manifolds, the set of admissible pairs for Strichartz estimates is much larger than usual. These results generalize the works on hyperbolic space due to Anker-Pierfelice and Ionescu-Staffilani. However, our approach is to employ the spectral measure estimates, obtained in the author’s joint work with Hassell, to establish the dispersive estimates for truncated / microlocalized Schr\"{o}dinger propagators as well as the corresponding energy estimates. Compared with hyperbolic space, the crucial point here is to cope with the conjugate points on the manifold. Additionally, these Strichartz estimates are applied to the $L^2$ well-posedness and $L^2$ scattering for nonlinear Schr\"{o}dinger equations with power-like nonlinearity and small Cauchy data.
\vspace{0.5cm}
\abstractin{french}

Dans cet article, nous examinons les estimations de Strichartz sans perte et globales en temps, d\'{e}finies sur des vari\'{e}t\'{e}s non-captives \`{a} asymptotiquement hyperboliques. De par la nature hyperbolique de ces vari\'{e}t\'{e}s, l’ensemble des paires admissibles pour les estimations de Strichartz est beaucoup plus grand que d’ordinaire. Ces r\'{e}sultats g\'{e}n\'{e}ralisent les travaux men\'{e}s par Anker-Pierfelice et Ionescu-Staffilani sur les espaces hyperboliques. Toutefois, notre approche utilise ici les estimations de mesures spectrales obtenues par l'autheur en collaboration avec Hassell afin d’\'{e}tablir les estimations de dispersion pour des propagateurs de Schr\"{o}dinger tronqu\'{e}s ou micro-localis\'{e}s et que les estimations des \'{e}nergies correspondantes. \`{A} la diff\'{e}rence des espaces hyperboliques, l’\'{e}l\'{e}ment crucial est ici de g\'{e}rer les points conjugu\'{e}s de la vari\'{e}t\'{e}. Enfin, ces estimations de Strichartz sont appliqu\'{e}es au caractère bien posé dans $L^2$ et à la diffusion $L^2$ pour les \'{e}quations de Schr\"{o}dinger avec des non-linearités de type puissance et des données initiales petites.

\end{abstract}

\maketitle

\section{Introduction}

This paper, following the author's joint works \cite{Chen-Hassell1} and \cite{Chen-Hassell2} with Andrew Hassell, is the last in a series of papers concerning the analysis of the resolvent family and spectral measure for the Laplacian on non-trapping asymptotically hyperbolic manifolds. The present paper is devoted to the application of spectral measure to Schr\"{o}dinger equations.

We investigate the Cauchy problem of the nonlinear Schr\"{o}dinger equation \begin{equation} \label{eqn:NLS} \left\{ \begin{array}{l} i \frac{\partial}{\partial t} u(t, z)  +  \Delta u(t, z) = F(u(t, z))  \\  u(0, z) = f(z) \end{array} \right. ,\end{equation} on an $n + 1$-dimensional asymptotically hyperbolic manifold $X$ (See Section \ref{sec : spectral measure} for the definition). Here the nonlinear term is power-like, i.e. $F$ satisfies $$|F(u)| \leq C |u|^\gamma \quad \mbox{and} \quad |F(u) - F(v)| \leq C (|u|^{\gamma - 1} + |v|^{\gamma - 1}) |u - v|,$$ for $1 < \gamma \leq 1 + 4/(n + 1)$. This sort of nonlinear dispersive equation cuts an important figure in mathematics and physics. A fundamental problem is the well-posedness of the solution. We say the equation \eqref{eqn:NLS} is globally well-posed in $L^2$ if for any subset $B$ of $L^2$ there exists a subspace $A$, continuously embedded into $C(\mathbb{R}_+; L^2(X))$ such that \eqref{eqn:NLS} has a unique solution in $A$ for any initial data $f$ and the map from $B$ to $A$ is continuous. Another interesting question is that how the solution of \eqref{eqn:NLS} behaves as time goes to infinity. One can understand that in terms of $L^2$ scattering, by which we mean the solution of \eqref{eqn:NLS} converges to the solution of the corresponding homogeneous linear equation in $L^2(X)$ sense as time goes $\pm \infty$. More precisely, for any solution $u$ of \eqref{eqn:NLS} there exists scattering data $u_{\pm}$ such that $$\big\|u - u_{\pm}\big\|_{L^2_z} \longrightarrow 0, \quad \mbox{as $t \rightarrow \pm$}.$$

By the classical theories of well-posedness and scattering for Schr\"{o}dinger equations (See for example \cite{Cazenave book, Tao book}), these problems usually reduce to the so-called 'Strichartz estimates' for the linear Schr\"{o}dinger equation \begin{equation}\label{schrodinger cauchy} \left\{ \begin{array}{l} i \frac{\partial}{\partial t} u  +  \Delta u = F(t, z)  \\  u(0, z) = f(z) \end{array} \right..\end{equation} It has been deeply studied on Euclidean space (See \cite{Kato, Ginibre-Velo-JFA-1995, Keel-Tao}) as well as on manifolds (See \cite{Bourgain-GAFA-1993-1, Bourgain-GAFA-1993-2, Burq-Gerard-Tzvetkov-2004, Burq-Gerard-Tzvetkov-2005}). We shall prove global-in-time Strichartz estimates without loss under the asymptotically hyperbolic regime.
 \begin{theorem}[Strichartz estimates]\label{inhomogeneous strichartz}Suppose $(X, g)$ is an $(n + 1)$-dimensional non-trapping asymptotically hyperbolic manifold with no resonance at the bottom of spectrum. For any admissible pairs $(q, r)$ and $(\tilde{q}^\prime, \tilde{r}^\prime)$ satisfying  \begin{equation}\label{schrodinger admissibility} \frac{2}{q} + \frac{n + 1}{r} \geq \frac{n + 1}{2}, \quad q \geq 2, \quad r > 2, \quad (q, r) \neq (2, \infty),\end{equation} we have the inhomogeneous Strichartz estimates \begin{equation}\label{strichartz inequality}\|u\|_{L^q_t L^r_z (\mathbb{R} \times X)} \leq C \big(\|f\|_{L^2(X)} + \|F\|_{L^{\tilde{q}^\prime}_t L_z^{\tilde{r}^\prime}(\mathbb{R} \times X)}\big),\end{equation} provided $f$and $F$ are both orthogonal to the eigenfunctions of $\Delta$.\end{theorem}

This result readily applies to $L^2$ well-posedness and $L^2$ scattering for the nonlinear Schr\"{o}dinger equation \eqref{eqn:NLS}. We obtain \begin{theorem}[$L^2$ Well-posedness and $L^2$ Scattering]\label{thm:well-posedness and scattering} Suppose $(X, g)$ is a manifold in Theorem \ref{inhomogeneous strichartz} and there are no eigenvalues of $\Delta$. Given $\gamma \in (1,  1 + 4/(n + 1)]$ and a small Cauchy data $f \in L^2(X)$, the nonlinear Schr\"{o}dinger equation \eqref{eqn:NLS} is globally well-posed in $L^2(X)$, whilst for any solution $u(t, z)$ there exists $u_\pm \in L^2(X)$ such that $$\|u(t) - e^{it\Delta} u_\pm\|_{L^2(X)} \longrightarrow 0, \quad \mbox{as $t \rightarrow \pm \infty$.}$$
\end{theorem} 

Assuming Theorem \ref{inhomogeneous strichartz}, the proof of Theorem \ref{thm:well-posedness and scattering} is standard. The well-posedness is given by a contraction mapping theorem method with the global-in-time Strichartz estimates, whilst the scattering part is simply given by a quick application of Cauchy criterion. We omit the proof and refer the reader to \cite{Anker-Pierfelice}, because their argument works verbatim in the asymptotically hyperbolic settings.

Return to the Strichartz estimates. By the classical approach formulated by Kato \cite{Kato},  Ginibre and Velo \cite{Ginibre-Velo-JFA-1995}, Keel and Tao \cite{Keel-Tao} etc., it is sufficient to prove energy estimates $$\|e^{i t \Delta} f\|_{L^2} \leq \|f\|_{L^2}$$ and dispersive estimates \begin{equation}\label{eqn:euclidean dispersive}\|e^{i(t - s)\Delta} f\|_{\infty} \leq |t - s|^{- (n + 1)/2} \|f\|_{L^1}\end{equation} for the Schr\"{o}dinger propagator on Euclidean space. However, unlike on Euclidean space, dispersive estimates in above form is too strong to hold on asymptotically hyperbolic manifolds.

First of all, unlike \eqref{eqn:euclidean dispersive}, we don't have dispersive estimates uniformly in time on hyperbolic space $\mathbb{H}^{n + 1}$ with $n > 2$. Actually, Anker and Pierfelice \cite{Anker-Pierfelice} independently Ionescu and Staffilani \cite{Ionescu-Staffilani-Mathann-2009} proved following dispersive estimates \begin{equation}\label{hyperbolic dispersive}\big|\text{Ker}\, e^{it\Delta_{\mathbb{H}^{n + 1}}} \big| \leq C \left\{ \begin{array}{cc} t^{-3/2} \big(1 + d(z, z^\prime)\big) e^{-n d(z, z^\prime) /2 } & \quad \mbox{if $t \geq 1 + d(z, z^\prime)$}\\  t^{- (n + 1)/2} \big(1 + d(z, z^\prime)\big)^{n/2} e^{-n d(z, z^\prime) /2 } & \quad \mbox{if $t \leq 1 + d(z, z^\prime)$}\end{array}\right.\end{equation} on real hyperbolic space $\mathbb{H}^{n + 1}$. Similar results on convex co-compact hyperbolic manifolds were proved by Burq, Guillarmou and Hassell \cite{Burq-Guillarmou-Hassell}. The spectral theorem gives $$e^{it\Delta} = e^{itn^2/4} \int_0^\infty e^{it\blambda^2} dE_{\sqrt{(\Delta - n^2/4)_+}}(\blambda, z, z').$$ So we can explain \eqref{hyperbolic dispersive} via the spectral measure $dE_{\sqrt{(\Delta - n^2/4)_+}}$. On the one hand, the nonuniformity of \eqref{hyperbolic dispersive} in time actually results from the discrepancy of powers in the spectral measure estimates on $\mathbb{H}^{n + 1}$. We shall see in Section \ref{sec : spectral measure}, Section \ref{sec : dispersive1} and \ref{sec : dispersive2} that the quicker growth for large $\blambda$ and the slower decay for small $\blambda$ creates the discrepancy of the powers of $t$. Consequently, the long time dispersive estimates on $\mathbb{H}^{n + 1}$ is at a lower speed than on Euclidean space, though the short time estimates are the same with Euclidean case. On the other hand, we also get something better. Near the spatial infinity, the spectral measure $dE_{\sqrt{(\Delta - n^2/4)_+}}$ as well as the Schr\"{o}dinger propagator $e^{it\Delta}$ gains an exponential decay factor. We can crudely interpret that as follows. Near the spatial infinity the conformal metric creates an exponentially growing volume. Since the spectral measure is globally $L^2$ integrable, it must decay exponentially to cancel the exponential growth volume as $d(z, z')$ goes to infinity. Not only does this long distance exponential decay compensate the low speed for long times but it also gives better Strichartz estimates. One may note a distinctive phenomenon that the admissible set \eqref{schrodinger admissibility} is much wider than the set of sharp Schr\"{o}dinger admissible pairs on Euclidean space, which satisfy\begin{equation*} \frac{2}{q} + \frac{n + 1}{r} = \frac{n + 1}{2}, \quad q, r \geq 2, \quad (q, r) \neq (2, \infty).\end{equation*} Banica, Carles and Staffilani \cite{Banica-Carles-Staffilani} first observed this while studying the radial solution of the Schr\"{o}dinger equations on $\mathbb{H}^{n + 1}$. Inspired by that, Anker and Pierfelice \cite{Anker-Pierfelice} independently Ionescu and Staffilani \cite{Ionescu-Staffilani-Mathann-2009} proved the Strichartz estimates on $\mathbb{H}^{n + 1}$ with such admissiblity.

More generally, if we study an asymptotically hyperbolic manifold with conjugate points, what kind of dispersive estimate could we get instead?  The Schwartz kernel of the spectral measure at high energy is a Lagrangian distribution microlocally supported on the geodesic flow-out. Then the difficulty is that there will be no global expression for the geodesic distance function. Because of that, neither \eqref{eqn:euclidean dispersive} nor \eqref{hyperbolic dispersive} will hold. Alternatively, one can microlocalize the spectral measure around the diagonal with a pair of pseudodifferential operators $(Q_k, Q_k^\ast)$.  Consequently we can get some near-diagonal estimates for the spectral measure, which give some sort of dispersive estimate for corresponding microlocalized Schr\"{o}dinger propagators. It is sufficient to establish the Strichartz estimates. Guillarmou, Hassell and Sikora \cite{Guillarmou-Hassell-Sikora}, Hassell and Zhang \cite{Hassell-Zhang}, applied this technique to the spectral measure and the Schr\"{o}dinger propagator on asymptotically conic (Euclidean) manifolds.

Due to above distinctive geometric and spectral properties on asymptotically hyperbolic manifolds, we will integrate some existing techniques to prove Theorem \ref{inhomogeneous strichartz}. First of all, we primarily follow the standard argument due to Kato, Ginibre-Velo and Keel-Tao. However, as mentioned before, their method doesn't exploit the distinctive phenomenons of hyperbolic type spaces, including the spatial decay at infinity and the low speed at long times. Therefore, borrowing the trick of Anker-Pierfelice and Ionescu-Staffilani, we split the time-space norm of the solution in temporal variables as well as the Schr\"{o}dinger propagator in spectral parameter. Also inspired by the microlocalization argument of Guillarmou-Hassell-Sikora and Hassell-Zhang, we establish the microlocalized dispersive estimates in Proposition \ref{dispersive estimate1} and in Proposition \ref{dispersive estimate2} to cope with the conjugate points.
 
The geometric microlocal technique used for the spectral measure does require the non-trapping condition and conformal compactness on the space. Bouclet \cite{Bouclet-apde-2011} investigates the local-in-time homogeneous Strichartz estimates without loss on more general asymptotically hyperbolic manifolds without taking these advantages. The author constructed a parametrix for Schr\"{o}dinger propagators. However, the issue here is that the error may be difficult to control as time goes to infinity. For the consideration of long time behaviour, one needs an exact spectral measure or propagator (a function of spectral measure) in these estimates. In the joint work of Hassell and the present author \cite{Chen-Hassell2}, we studied the spectral measure estimates on asymptotically hyperbolic manifolds, which enables us to study the global-in-time Strichartz estimates.

The paper is organized as follows. First of all, we shall review the asymptotically hyperbolic manifold and the spectral measure in Section \ref{sec : spectral measure}. Based on the spectral measure estimates, we introduce the microlocalized / truncated expressions of Schr\"{o}dinger propagators. We then turn to the proof of $L^2$-energy estimates. In Section \ref{sec : dispersive1} and Section \ref{sec : dispersive2}, we establish the dispersive estimates for microlocalized / truncated propagators. The Strichartz estimates will be proved in the last two sections.

\section{Spectral measure on asymptotically hyperbolic manifolds}\label{sec : spectral measure}

A conformally compact manifold $X$ is an $(n + 1)$-dimensional manifold with boundary $\partial X$, compact closure $\bar{X}$ and endowed with a Riemannian metric $g$ which extends smoothly to its closure. One can write $$g = \frac{dx^2}{x^2} + \frac{h(x, y, dx, dy)}{x^2},$$ where $x$ is a boundary defining function, and $h$ is a metric on the boundary but depending parametrically on $x$. Mazzeo \cite{Mazzeo-JDG-1988} showed $g$ is complete and its sectional curvature approaches $- |dx|^2_{x^2 g}$ as it approaches the boundary. In particular, $g$ is said to be asymptotically hyperbolic if $- |dx|^2_{x^2 g} = - 1$.

Let $\Delta$ be the Laplacian, on $(n+1)$-dimensional non-trapping asymptotically hyperbolic manifold $X$. As on $\mathbb{H}^{n + 1}$, the continuous spectrum of $\Delta$ is contained in $[n^2/4 , \infty)$. Additionally, Mazzeo \cite{Mazzeo-1991} showed the point spectrum is contained in $(0, n^2/4)$.

Mazzeo and Melrose \cite{Mazzeo-Melrose} constructed the resolvent $(\Delta - \bsigma(n - \bsigma))^{-1}$ on asymptotically hyperbolic manifolds for fixed parameter $\bsigma$ and proved it has a meromorphic extension except at points $(n + 1)/2 - \mathbb{Z}_+$. The resolvent they construct is a $0$-pseudo differential operator plus a smooth function on the $0$-blown up double space $X \times_0 X$ ( or $X_0^2$ for short), where the space $X \times_0 X$ is obtained by blowing up the boundary of the diagonal $\partial \text{diag} = \{(0, y, 0, y)\} \in X^2$. \begin{center}\includegraphics[width=0.5\textwidth]{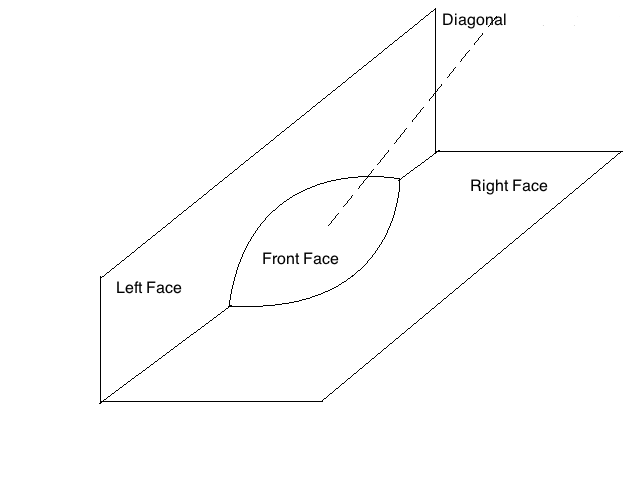}\end{center} From here on, we will work on $X_0^2$ instead of $X^2$ for the nice expression of the resolvent and the spectral measure. A very important feature of $X_0^2$ is that the front face is a bundle with fibres similar to hyperbolic space. Therefore hyperbolic space is a good model for asymptotically hyperbolic manifolds.  Apart from these, we also get following useful asymptotic expansion of the geodesic distance function near the boundaries of $X_0^2$  \begin{equation*} d(z, z') = - \log \rho_L - \log \rho_R + b(z, z'),\end{equation*} where $\rho_L$ and $\rho_R$ are boundary defining functions of the left and the right faces respectively, $b$ is a uniformly bounded function on $X_0^2$. See \cite[Proposition 10]{Chen-Hassell2}. In particular, $b(z, z')$ is smooth on $X_0^2$ in the case of asymptotically hyperbolic manifolds of Cartan-Hadamard type. This was observed on hyperbolic spaces and proved on asymptotically hyperbolic manifolds by  Melrose, S\'{a} Barreto and Vasy \cite{Melrose-Sa Barreto-Vasy}. On general asymptotically hyperbolic manifolds, arising conjugate points ruin the smoothness of $b$ but we still get the boundedness of $b$. Consequently, we have the asymptotic \begin{equation}
\label{asymptotic of distance function} e^{- d(z, z')} \approx \rho_L \rho_R, \quad \mbox{if $\rho_L \rho_R$ is small.}
\end{equation}

Additionally, Melrose, S\'{a} Barreto and Vasy \cite{Melrose-Sa Barreto-Vasy}, Wang \cite{Wang}, and Hassell and the present author \cite{Chen-Hassell1} constructed the semiclassical resolvent at high energy (near the infinity of the spectrum). Specifically, the high energy resolvent defined on $X_0^2$ is a $0$-pseudo differential operator plus a Fourier integral operator microlocally supported on the union of the diagonal conormal bundle and its bicharacteristic flow-out.  To avoid unnecessary technical details, we wouldn't repeat the bulky theories about $0$-calculus, blow-up, flow-out, Lagrangian distribution, intersecting Lagrangian and etc, but refer the readers to \cite{Mazzeo-Melrose, Melrose-Sa Barreto-Vasy, Chen-Hassell1}.

Based on the results of the resolvent,  Hassell and the author \cite{Chen-Hassell2} studied the spectral measure on asymptotically hyperbolic manifolds, via Stone's formula $$2 \pi i \, dE_{L} (\blambda) = R_{L}(\blambda + i0) - R_L(\blambda - i0),$$ provided $\blambda$ is in the continuous spectrum of $L$.\footnote{We use Greek letters $\lambda, \mu, \zeta$ to denote the phase variables on cotangent bundle, respectively bold Greek letters $\blambda, \bmu, \bzeta$ to denote spectral parameters.} Since the spectral measure is defined on the continuous spectrum $(n^2/4, \infty)$, we are in particular concerned about the asymptotic behaviour around two endpoints $n^2/4$ (low energy) and $\infty$ (high energy) respectively. Because of the absence of imbedded eigenvalues, the intermediate values can be estimated in either way.

On the one hand, the spectral measure $dE_P(\blambda)$\footnote{Here we use another spectral parameter $\blambda \in [0, \infty)$ with $\bsigma = n/2 \pm i \blambda$.} with $P = \sqrt{(\Delta - n^2/4)_+}$ for small $\blambda$ has a similar structure to the resolvent near the bottom of the spectrum. As on $\mathbb{H}^{n + 1}$, it is convenient to assume the smoothness of the resolvent at the bottom of the spectrum to gain the asymptotic of spectral measure. We say there is no resonance at the bottom of the continuous spectrum if the resolvent is analytic at $n^2/4$. \footnote{Intriguingly, it is still unknown that what geometric conditions amounts to the analyticity of the resolvent at the bottom of spectrum. However, there are some sufficiency results. For instance Guillarmou and Qing \cite{Guillarmou-Qing} shows that the largest real scattering pole of $(\Delta - \bsigma(n - \bsigma))^{-1}$ on an $n + 1$-dimensional conformally compact Einstein manifold $(X, g)$ is less than $n/2 - 1$ if and only if the conformal infinity of $(X, g)$ is of positive Yamabe type, where $n > 1$.} With this hypotheses of analyticity at $n^2/4$, we \cite{Chen-Hassell2}  deduce, from the resolvent of Mazzeo and Melrose, that \begin{equation}\label{low energy spectral measure}dE_{P}(\blambda)(z, z^\prime) = \blambda \Big((\rho_L\rho_R)^{n/2 + i\blambda} a(\blambda, z, z')  -  (\rho_L\rho_R)^{n/2 - i\blambda} a(- \blambda, z, z')\Big) \quad \mbox{when $\blambda < 1$ },\end{equation} where $P = \sqrt{(\Delta - n^2/4)_+}$ and $a \in C^\infty([0, 1]_{\blambda^{-1}} \times X_0^2)$. 
A quick corollary of this result is that \begin{equation}\label{low spectral measure upper bound} |dE_P(\blambda)(z, z')| \leq C \blambda^2 (1 + d(z, z')) e^{-nd(z, z')/2}.\end{equation} One may note the spectral measure doesn't vanish as rapidly as its counterparts on some other space do at low energy. For example, Guillarmou, Hassell and Sikora \cite{Guillarmou-Hassell-Sikora}, Hassell and Zhang \cite{Hassell-Zhang} showed on $n + 1$-dimensional asymptotically Euclidean manifolds there is a pseudodifferential operator partition of unity $$I = \sum_{i = 1}^N Q_i$$ such that the microlocalized spectral measure   reads \begin{equation}\label{eqn:spectral measure conic}Q_i dE_{P} Q_i^\ast(\blambda, z, z') = \blambda^n e^{i \blambda d(z, z^\prime)} a_+(\blambda, z, z^\prime) + \blambda^n e^{- i \blambda d(z, z^\prime)} a_-(\blambda, z, z^\prime),\end{equation} where the derivatives of $a_{\pm}$ obeys $$\bigg|\frac{d^\alpha}{d\blambda^\alpha} a_{\pm}(\blambda, z, z^\prime)\bigg| \leq C \blambda^{-\alpha} \big(1 + \blambda d(z, z^\prime)\big)^{-n/2}.$$ We remark that one could remove the diagonal microlocalization $(Q_i, Q_i^\ast)$ in case there are no conjugate points on the manifold; however it is necessary for general settings. Apart from the $3$-dimensional space where $n = 2$, the spectral measure on asymptotically hyperbolic manifolds is unable to provide such decay for low energies. Nonetheless, the property \eqref{asymptotic of distance function} for large distance on asymptotically hyperbolic manifolds compensates the lack of decay with an exponential vanishing at spatial infinity.

On the other hand, the spectral measure for large $\blambda$ shares a microlocal structure with the resolvent at high energies. Suppose we have local coordinates $\{(x, y_1, \dots, y_n)\}$ near $\partial X$ and local coordinates $\{(z_1, \dots, z_{n+1})\}$  away from $\partial X$. The $0$-cotangent bundle ${}^0 T^\ast X^\circ$, introduced by Mazzeo and Melrose \cite{Mazzeo-Melrose}, is a vector bundle with sections \begin{eqnarray*}\lambda \frac{dx}{x} + \mu_1 \frac{dy_1}{x} + \cdots + \mu_n \frac{dy_n}{x} && \mbox{near $\partial X$}\\ \zeta_1 \frac{dz_1}{x} + \cdots + \zeta_{n + 1} \frac{dz_{n + 1}}{x}  && \mbox{away from $\partial X.$} \end{eqnarray*} Recall from \cite{Chen-Hassell1} that the microlocal support (or wavefront set) of the high energy resolvent is the diagonal conormal bundle $N^\ast \text{diag} \subset {}^0 T X_0^2$ and its bicharacteristic flow-out $\Lambda$, which is contained in ${}^0 S X^\circ \times {}^0 S X^\circ$, where $${}^0 S^\ast X^\circ = \{|\zeta|^2 = 1 \, \mbox{or} \, |\lambda|^2 + |\mu|^2 = 1 \} \subset {}^0TX^\circ.$$  By Stone's formula, the spectral measure is microlocally supported on $\Lambda$, while the singularity at $N^\ast \text{diag}$ cancels out by the subtraction between the outgoing resolvent and the incoming resolvent. Therefore the spectral measure is a Fourier integral operator associated with the Lagrangian $\Lambda$. Apart from the boundary behaviour, this Lagrangian structure on asymptotically hyperbolic manifolds is analogous with the case of asymptotically Euclidean. So we can gain similar spectral measure estimates at high energies to \eqref{eqn:spectral measure conic}.

To state the microlocalized spectral measure estimates explicitly, let us recall the partition of unity on ${}^0 T^\ast X^\circ$ in \cite{Chen-Hassell2}.   First of all, we take $Q_0$ microlocally supported away from the spherical bundle ${}^0 S^\ast X^\circ$, say $\{|\zeta|^2 > 3/2 \, \mbox{or} \, |\lambda|^2 + |\mu|^2 > 3/2\}$, which contains the wavefront set of the spectral measure. On the other hand, we divide the interval $(-3/2, 3/2)$ into a union of intervals $I_1, \dots, I_{N_1}$ with overlapping interiors, and
with diameter $\leq \delta$, which is a sufficiently small number, whilst each $I_i$ intersects only $I_{i - 1}$ and $I_{i + 1}$. We also take a small strip neighbourhood of the boundary such that the sectional curvature is negative; in the meantime, we divide the $0$-cotangent bundle over this strip into a union of small slices $B_1, \dots, B_{N_1}$ such that every $B_i \subset \{ \lambda \in I_i \}$. Then we have $0$th-order pseudodifferential operators $Q_1, \dots, Q_{N_1}$ supported on them repectively. Next, we divide the remaining region into the union of small balls $B_{N_1 + 1}, \dots B_{N_2}$ with diameter $\leq \eta$, which is also sufficiently small, and have $Q_{N_1 + 1}, \dots Q_{N_2}$ supported on them respectively.  With this partition, we have the estimates for the microlocalized spectral measure. \begin{proposition}[\cite{Chen-Hassell2}]\label{derivative spectral measure estimate} One can choose a pseudodifferential operator partition of unity $$Id = \sum_{k = 0}^{N} Q_k(\blambda),$$ where $Q_k$ for $k \neq 0$ is supported around the spherical bundle, such that $Q_k$ for any $k$ are uniformly ($L^2$-)bounded over $\blambda$ and \begin{eqnarray*} Q_k(\blambda)  dE_{P}(\blambda) Q_k^\ast (\blambda) = \blambda^n e^{i\blambda d(z, z')} a_+(\blambda) + \blambda^n e^{- i\blambda d(z, z')} a_-(\blambda) + O(\blambda^{- \infty}), \quad \mbox{for large $\blambda,$}\end{eqnarray*} where $a_\pm$ are defined on the forward and backward bicharacteristic flow respectively and satisfying \begin{eqnarray*}\frac{d^j}{d\blambda^j}  a_\pm(\blambda) = \left\{ \begin{array}{l@{\quad , \quad}l} O\Big(\blambda^{ - j}\big(1 + \blambda d(z, z^\prime)\big)^{- n/2} \Big)  & \mbox{if $d(z, z^\prime)$ is small}\\  O\Big(\blambda^{- n/2 - j} e^{- n d(z, z^\prime) / 2}\Big) & \mbox{if $d(z, z^\prime)$ is large}\end{array}\right..\end{eqnarray*} \end{proposition} This result is actually better than \eqref{eqn:spectral measure conic} for large $\blambda$. Not only does it give the same growth rate in $\blambda$, but there is also a spatial exponential decay.

Moreover, the restriction theorem (in the sense of Stein and Tomas) $$\|dE_{P}(\blambda)\|_{L^p \rightarrow L^{p^\prime}} \leq C \blambda^{(n + 1) (1/p - 1/p^\prime) - 1} \quad \mbox{where $p \in [1, 2(n + 2)/(n + 4)]$},$$ at high energies on non-trapping asymptotically hyperbolic manifolds follows from above spectral measure estimates. It is well-known that Strichartz \cite{Strichartz} insightfully points out the deep relationship between Strichartz estimates and restriction theorem. It motivates us to show the Strichartz estimates from these spectral measure estimates, which are sufficient to give restriction theorem. In fact, combining Strichartz estimates and dispersive estimates in this paper with our previous results of resolvent in \cite{Chen-Hassell1}, spectral measure with applications to restriction theorem and spectral multiplier in \cite{Chen-Hassell2}, we have elucidated the following diagram on non-trapping asymptotically hyperbolic manifolds. $$\begin{array}{ccccc}
&&\begin{array}{c}\mbox{$L^p$ boundedness} \\ \mbox{of spectral multiplier}\end{array} &&\\ &&\uparrow && \\ \begin{array}{c}\mbox{Strichartz estimate}\\ \mbox{of Schr\"{o}dinger equation}\end{array}&\leftarrow&\begin{array}{c}\mbox{Restriction theorem}\\ \mbox{of spectral measure}\end{array} &&  \\ \uparrow&&\uparrow &&\\ \begin{array}{c}\mbox{Dispersive estimate}\\ \mbox{of Schr\"{o}dinger propagator}\end{array}&\leftarrow&\begin{array}{c}\mbox{Pointwise estimates} \\ \mbox{for spectral measure}\end{array} &&\\ &&\uparrow && \\ &&\begin{array}{c}\mbox{Resolvent construction} \\ \mbox{near continuous spectrum}\end{array} &&
\end{array}$$

\section{Schr\"{o}dinger propagators via spectral measure}\label{sec : propagator decomposition}

The spectral theorem of projection valued measure form for unbounded self-adjoint operators gives following expression of Schr\"{o}dinger propagators $e^{i t \Delta}$ via spectral measure, $$e^{i t \Delta} = e^{i t  n^2/4}\int_0^\infty e^{i t \blambda^2} dE_{P}(\blambda).\footnote{Since we assume $f$ and $F$ are orthorgonal to the eigenfunction spaces, the discrete terms don't show up.}$$ 

We are motivated to employ the spectral measure estimates, including the microlocalized form at high energy (say $\blambda > 1$) together with the global form at low energy (say $\blambda < 1$), to estimate the Schr\"{o}dinger propagator. As seen on Euclidean space or asymptotically conic manifolds, the spectral measure behaves uniformly on the full continuous spectrum, for example see \cite{Guillarmou-Hassell-Sikora}. However, comparing \eqref{low energy spectral measure} and Proposition \ref{derivative spectral measure estimate}, one can see that the discrepancy of the order of $\blambda$ between low energies and high energies on $n + 1$-dimensional asymptotically hyperbolic manifolds for $n + 1 > 3$. We thus have to split up the propagator to remedy the discrepancy.

One may pick two smooth bump functions $\chi_{\text{low}}$ supported in $[0, 2)$ and $\chi_\infty$ supported in $(1, \infty)$ such that $\chi_{\text{low}} + \chi_\infty = 1$  and split the propagator as $$U(t) = \int_0^\infty e^{it\blambda^2} \chi_{\text{low}}(\blambda) \, dE_{P}(\blambda) + \int_0^\infty e^{it\blambda^2} \chi_\infty(\blambda) \, dE_{P}(\blambda).$$ In accordance with Proposition \ref{derivative spectral measure estimate}, we also will have to microlocalize the spectral measure at high energies by a family of semiclassical pseudodifferential operators $\{Q_k\}_{0}^{N}$ as follows $$U_k = \int_0^\infty e^{it\blambda^2} \chi_\infty(\blambda) Q_k(\blambda) \, dE_{P}(\blambda).$$ In summary, we truncate and microlocalize the propagator and gain following decomposition $$e^{it\Delta - it n^2/4} = U_{\text{low}} + \sum_{k = 0}^N U_k.$$ 

Returning to the Cauchy problem \eqref{schrodinger cauchy}, the solution $u$ is given by Duhamel's formula $$u(t, z) = e^{it \Delta}f(z) - i \int_0^t e^{i(t - s)\Delta} F(s, z) \,ds.$$ To prove the Strichartz estimates, we shall invoke Keel-Tao bilinear approach. In our case, we reduce to the energy estimates and dispersive estimates for following bilinear propagators \begin{eqnarray}\label{eqn:prpgt1}U_{\text{low}}(t) U_{\text{low}}^\ast(s) &=& \int_0^\infty e^{i(t - s) \blambda} \chi_{\text{low}} dE_P(\blambda),\\ \label{eqn:prpgt2} U_{k}(t) U_{k}^\ast(s) &=& \int_0^\infty e^{i(t - s) \blambda} \chi_{\infty} Q_k(\blambda) dE_P(\blambda)Q_k^\ast(\blambda),\\ \label{eqn:prpgt3} U_j(t) U_k^\ast(s) &=& \int_0^\infty e^{i(t - s) \blambda} \chi_{\infty} Q_j(\blambda) dE_P(\blambda)Q_k^\ast(\blambda). \end{eqnarray} Aside from these, we also remark the estimates for $U_{\text{low}}(t) U_{k}^\ast(s)$ or $U_{k}(t) U_{\text{low}}^\ast(s)$ are the same with $U_{\text{low}}(t) U_{\text{low}}^\ast(s)$.

In the next three sections, we prove the energy estimates for them in Proposition \ref{energy estimate} and dispersive estimates for \eqref{eqn:prpgt1} and \eqref{eqn:prpgt2} in Proposition \ref{dispersive estimate1} and for \eqref{eqn:prpgt3} in Proposition \ref{dispersive estimate2} respectively.

\section{Energy estimates for  Schr\"{o}dinger propagators at high energy}

We shall prove the $L^2$-boundedness of microlocalized / truncated Schr\"{o}dinger propagators. More precisely, \begin{proposition}[Energy estimates]\label{energy estimate}The propagator $e^{it\Delta}$, low energy truncated propagator $U_{\text{low}}$, microlocalized high energy truncated propagators $U_0$ and $U_k$ for $k = 1, 2, \dots$ are all $L^2$-bounded\end{proposition}

\begin{proof}\footnote{This proof is essentially due to Hassell and Zhang \cite{Hassell-Zhang} in case of asymptotically Euclidean manifolds, as only minor modifications are needed here. But we give the detailed proof for the self-containedness of the paper. }

 The boundedness of $e^{it\Delta}$ and $U_{\text{low}}$ is clear. Since the entire cut off propagator at high energy is of course $L^2$-bounded, we can ignore the $k = 0$ term but only consider $U_k$ for $k = 1, 2, \dots.$

Our main tool is almost orthogonality lemma established by Cotlar, Knapp and Stein, see for example \cite[p. 620]{fourier analysis}. \begin{lemma}[Almost orthogonality]\label{almost orthogonality}Let $\{T_j\}_{j \in \mathbb{Z}}$ be a family of bounded operators on Hilbert space $H$ obeying $$\|T_j^\ast T_k\|_{H \rightarrow H} + \|T_j T_k^\ast\|_{H \rightarrow H} \leq \gamma(j - k) \quad \mbox{for any $j, k \in \mathbb{Z}$},$$ where the function $\gamma : \mathbb{Z} \rightarrow \mathbb{R}^{+}$ satisfies $\sum_{j \in \mathbb{Z}} \sqrt{\gamma(j)} < \infty$.Then linear operator $T$, the limit of $\sum_{|j| < N} T_j$ in the norm topology of $H$ as $N$ goes to infinity, is $H$-bounded.\end{lemma}

First of all, the propagators are well-defined on $L^2$ if the integrand is supported on a compact subset of $(0, \infty)$ in $\blambda$ as the pseudodifferential operator would be $L^2$-bounded uniformly with respect to $\blambda$. We want to extend the well-definedness to entire positive half real line by almost orthogonality.

The strategy is to get a decomposition of the microlocalized propagator such that every term is an integral of a compactly supported function with respect to the microlocalized spectral measure, and then show the almost orthogonality of the decomposition required in Lemma \ref{almost orthogonality}.

First of all, we take the decomposition with a compactly supported smooth function $\psi \in C_c^\infty [1/2, 2]$ valued in $[0, 1]$ such that $$\sum_j \psi\bigg(\frac{\blambda}{2^j}\bigg) = 1.$$ Then we define \begin{eqnarray*}U_{i, j}(t) &=& \int_0^\infty e^{it\blambda^2}  \chi_\infty(\blambda) \psi\bigg(\frac{\blambda}{2^j}\bigg) Q_i(\blambda) \, dE_{P}(\blambda)\\& = &- \int_0^\infty \frac{d}{d\blambda}\bigg(e^{it\blambda^2} \chi_\infty(\blambda) \psi\bigg(\frac{\blambda}{2^j}\bigg)Q_i(\blambda)\bigg) E_{P}(\blambda).\end{eqnarray*} and calculate as follows \begin{eqnarray*}U_{i, j}(t) U_{i, k}^\ast(t)&=& \int\!\!\!\int \frac{d}{d\blambda} \bigg( e^{it\blambda^2} \chi_\infty(\blambda) \psi\bigg(\frac{\blambda}{2^j}\bigg) Q_i(\blambda)\bigg) E_{P}(\blambda) \\&& \quad\quad\quad\times E_{P}(\bmu) \frac{d}{d\bmu} \bigg( e^{-it\bmu^2} \chi_\infty(\bmu) \psi\bigg(\frac{\bmu}{2^k}\bigg) Q_i^\ast(\bmu)\bigg) \, d\blambda d\bmu \\&=& \int\!\!\!\int_{\blambda \leq \bmu} \frac{d}{d\blambda} \bigg( e^{it\blambda^2} \chi_\infty(\blambda) \chi_\infty(\blambda) \psi\bigg(\frac{\blambda}{2^j}\bigg) Q_i(\blambda)\bigg) E_{P}(\blambda)\\&& \quad\quad\quad  \times\frac{d}{d\bmu} \bigg( e^{-it\bmu^2} \chi_\infty(\bmu) \psi\bigg(\frac{\bmu}{2^k}\bigg) Q_i^\ast(\bmu)\bigg) \, d\blambda d\bmu \\&& + \int\!\!\!\int_{\bmu \leq \blambda} \frac{d}{d\blambda} \bigg( e^{it\blambda^2} \chi_\infty(\blambda) \psi\bigg(\frac{\blambda}{2^j}\bigg) Q_i(\blambda)\bigg)  E_{P}(\bmu)\\&& \quad\quad\quad
\quad \times \frac{d}{d\bmu} \bigg( e^{-it\bmu^2} \chi_\infty(\bmu) \psi\bigg(\frac{\bmu}{2^k}\bigg) Q_i^\ast(\bmu)\bigg) \, d\blambda d\bmu \end{eqnarray*} We then perform integration by parts and get \begin{eqnarray*}U_{i, j}(t) U_{i, k}^\ast(t) &=& \int \frac{d}{d\blambda} \bigg( e^{it\blambda^2} \chi_\infty(\blambda) \psi\bigg(\frac{\blambda}{2^j}\bigg) Q_i(\blambda)\bigg) E_{P}(\blambda) \\&& \quad\quad\quad \times\bigg(- e^{-it\blambda^2} \chi_\infty(\blambda) \psi\bigg(\frac{\blambda}{2^k}\bigg) Q_i^\ast(\blambda)\bigg) \, d\blambda  \\&& + \int \bigg(- e^{it\bmu^2} \chi_\infty(\bmu) \psi\bigg(\frac{\bmu}{2^j}\bigg) Q_i(\bmu)\bigg)  E_{P}(\bmu)\\&& \quad\quad\quad \times\frac{d}{d\bmu} \bigg( e^{-it\bmu^2} \chi_\infty(\bmu) \psi\bigg(\frac{\bmu}{2^k}\bigg) Q_i^\ast(\bmu)\bigg) \,  d\bmu \\ &=& \int  \chi_\infty^2(\blambda) \psi\bigg(\frac{\blambda}{2^j}\bigg)\psi\bigg(\frac{\blambda}{2^k}\bigg) Q_i(\blambda) dE_{P}(\blambda)Q_i^\ast(\blambda) \\ &=& \int \frac{d}{d\blambda}\bigg( \chi_\infty^2(\blambda) \psi\bigg(\frac{\blambda}{2^j}\bigg)\psi\bigg(\frac{\blambda}{2^k}\bigg) Q_i(\blambda)\bigg) E_{P}(\blambda)Q_i^\ast(\blambda)\\ && + \int  \chi_\infty^2(\blambda) \psi\bigg(\frac{\blambda}{2^j}\bigg)\psi\bigg(\frac{\blambda}{2^k}\bigg) Q_i(\blambda) E_{P}(\blambda) \frac{d}{d\blambda}Q_i^\ast(\blambda) . \end{eqnarray*}

As implied, $U_{i, j}(t) U_{i, k}^\ast(t)$ is indeed $t$-independent. Therefore, we shall prove the $L^2$-boundedness for all $t$ via $U_{i, j}(0) U_{i, k}^\ast(0)$, which equals $$\int\!\!\!\int E_{P}(\blambda) \frac{d}{d\blambda}\bigg(\chi_\infty(\blambda) \psi\bigg(\frac{\blambda}{2^j}\bigg)Q_i^\ast(\blambda)\bigg) \frac{d}{d\bmu}\bigg(Q_i(\bmu)\psi\bigg(\frac{\bmu}{2^k}\bigg)\chi_\infty(\bmu)\bigg) E_{P}(\bmu) \, d\blambda d\bmu.$$

We claim $U_{i, j}(0) U_{i, k}^\ast(0)$ obeys the almost orthogonality estimate $$\|U_{i, j}(0) U_{i, k}^\ast(0)\|_{L^2 \rightarrow L^2} \leq C 2^{- |j - k|}.$$ In light of the $L^2$-boundedness of spectral projection, it suffices to prove $$ \frac{d}{d\blambda}\bigg(\chi_\infty(\blambda) \psi\bigg(\frac{\blambda}{2^j}\bigg)Q_i^\ast(\blambda)\bigg) \frac{d}{d\bmu}\bigg(Q_i(\bmu)\psi\bigg(\frac{\bmu}{2^k}\bigg)\chi_\infty(\bmu)\bigg) \leq C 2^{- |j - k|}.$$ We denote the operators in the parentheses $Q^\ast_{i, j}(\blambda)$ and $Q_{i, k}(\bmu)$ respectively. We write the product of the two as \begin{eqnarray*}Q_{i, j}^\ast(\blambda) Q_{i, k}(\bmu)&=&  \blambda^{n + 1} \bmu^{n + 1} \int\!\!\!\int\!\!\!\int e^{i \blambda (z - z^{\prime\prime}) \cdot \zeta / x''} q_{i, j}(z^{\prime\prime}, \zeta, \blambda)\\&& \quad\quad\quad \times e^{i \bmu (z^{\prime\prime} - z^\prime) \cdot \zeta^\prime / x''} q_{i, k}(z^{\prime\prime}, \zeta^\prime, \bmu) \, d\zeta d\zeta^\prime dz^{\prime\prime},\end{eqnarray*} away from $\partial X$ or
\begin{eqnarray*}  \begin{gathered}Q_{i, j}^\ast(\blambda) Q_{i, k}(\bmu)  = \blambda^{n + 1} \bmu^{n + 1} \int\!\!\!\int\!\!\!\int e^{i \blambda ( (x - x'')  \lambda + (y - y'') \cdot \mu )/ x''} q_{i, j}(x'', y'', \lambda, \mu, \blambda) \\e^{i \bmu ((x'' - x^\prime)  \lambda' +  (y'' - y') \cdot \mu' ) / x''} q_{i, k}(x'', y'', \lambda', \mu', \bmu) \, d\lambda d\mu d\lambda' d\mu' dx'' dy'',\end{gathered}\end{eqnarray*} near $\partial X$. The second case is indeed the same with the first, as one can denote $(x, y)$ by $(z_1, \dots, z_n$ and $(\lambda, \mu)$ by $(\zeta_1, \dots, \zeta_n)$.
Furthermore, one may assume $j > k$, equivalent to $\blambda > \bmu$, due to the symmetry. We insert a differential operator $i x'' \zeta \cdot \partial_{z^{\prime \prime}} / (\blambda |\zeta|^2)$, to which $e^{i \blambda (z - z^{\prime\prime}) \cdot \zeta /x''}$ is invariant,  and take integration by parts. \begin{eqnarray*}
\lefteqn{\blambda^{- n - 1} \bmu^{- n - 1}  Q_{i, j}^\ast(\blambda) Q_{i, k}(\bmu)} \\&=& \int\!\!\!\int\!\!\!\int \frac{i x'' \zeta \cdot \partial_{z^{\prime \prime}}}{\blambda |\zeta|^2} \Big(e^{i \blambda (z - z^{\prime\prime}) \cdot \zeta / x''} \Big)  q_{i, j}(z^{\prime\prime}, \zeta, \blambda) e^{i \bmu (z^{\prime\prime} - z^\prime) \cdot \zeta^\prime / x''} q_{i, k}(z^{\prime\prime}, \zeta^\prime, \bmu) \, d\zeta d\zeta^\prime dz^{\prime\prime}  \\&=&  \frac{\bmu}{\blambda}  \int\!\!\!\int\!\!\!\int  e^{i \blambda (z - z^{\prime\prime}) \cdot \zeta / x''}  \frac{ \zeta \cdot \zeta^\prime}{ |\zeta|^2} e^{i \bmu (z^{\prime\prime} - z^\prime) \cdot \zeta^\prime / x''} q_{i, j}(z^{\prime\prime}, \zeta, \blambda)  q_{i, k}(z^{\prime\prime}, \zeta^\prime, \bmu) \, d\zeta d\zeta^\prime dz^{\prime\prime} \\&& -  \frac{x''}{\blambda}   \int\!\!\!\int\!\!\!\int i e^{i \blambda (z - z^{\prime\prime}) \cdot \zeta / x''} e^{i \bmu (z^{\prime\prime} - z^\prime) \cdot \zeta^\prime / x''}  \frac{ \zeta}{ |\zeta|^2}  \cdot \partial_{z^{\prime\prime}} \bigg(q_{i, j}(z^{\prime\prime}, \zeta, \blambda)  q_{i, k}(z^{\prime\prime}, \zeta^\prime, \bmu) \bigg)   \, d\zeta d\zeta^\prime dz^{\prime\prime}\end{eqnarray*} Because $i \neq 0$, $Q_i$ is microlocally supported around the spherical bundle, namely, $|\zeta| \approx |\zeta^\prime| \approx 1$. Therefore, using the $L^2$-boundedness of semiclassical pseudodifferential operators and noting $\blambda, \bmu \geq 1$ on the support of the high energy cut-off function $\chi_\infty$, we deduce $$\big\|Q_{i, j}^\ast(\blambda) Q_{i, k}(\bmu)\big\|_{L^2 \rightarrow L^2} \leq C \frac{\bmu + x''}{\blambda}  \leq C \frac{\bmu}{\blambda} \leq C 2^{- {|j - k|}},$$ which proves the almost orthogonality for $t = 0$. Almost orthogonality lemma then gives that $\sum_{|j| \leq l} U_{i, j}^\ast(0)$ strongly converges in $L^2$, that is, $$\lim_{l \rightarrow \infty} \sup_{m > l} \bigg\|\sum_{l \leq |j| \leq m} U_{i, j}^\ast(0) f\bigg\|^2_{L^2} = 0.$$

We now extend this conclusion to any $t$. Given $f \in L^2$, we want to have $$\lim_{l \rightarrow \infty} \sup_{m > l} \bigg\| \sum_{l \leq |j| \leq m} U_{i, j}^\ast(t) f \bigg\|_{L^2}^2 = \lim_{l \rightarrow \infty} \sup_{m > l}  \sum_{l \leq |j|, |j^\prime| \leq m} \langle U_{i, j}(t)U_{i, j^\prime}^\ast(t) f , f \rangle = 0.$$
In fact, it is easy to reduce the convergence for general $t$ to the case $t = 0$ by the time independence of the operator $U_{i, j}(t)U^\ast_{i, j}(t)$
\begin{eqnarray*} \lim_{l \rightarrow \infty} \sup_{m > l} \sum_{l \leq |j|, |j^\prime| \leq m} \langle  U_{i, j}(t)U_{i, j^\prime}^\ast(t) f , f \rangle = \lim_{l \rightarrow \infty} \sup_{m > l} \sum_{l \leq |j|, |j^\prime| \leq m} \langle  U_{i, j}(0)U_{i, j^\prime}^\ast(0) f , f \rangle = 0\end{eqnarray*} Finally, noting $$\|U_i^\ast(t)\|^2 \leq\lim_{l \rightarrow \infty}\bigg\|\sum_{|j| \leq l}  U_{i, j}^\ast(t)\bigg\|^2,$$ we conclude that $U_i(t)$ is uniformly bounded on $L^2$.

\end{proof}

\section{Dispersive estimates for Schr\"{o}dinger propagators I}\label{sec : dispersive1}

In this section we establish the dispersive estimates for diagonal microlocalized / truncated Schr\"{o}dinger propagators. We shall show \begin{proposition}[Dispersive estimates I]\label{dispersive estimate1} The long time dispersive estimates for the microlocalized Schr\"{o}dinger propagators at high energy \begin{eqnarray} \label{long time dispersive microlocalized} &&\begin{gathered} \bigg|  \int_0^\infty    e^{i t \blambda^2} \chi_\infty \Big(Q_k(\blambda) d E_{P}(\blambda) Q_k^\ast (\blambda)\Big) (z, z^\prime) \, d\blambda \bigg| \leq C |t|^{- \infty}  e^{- n d(z, z^\prime)/2} \end{gathered}\end{eqnarray} hold, provided $t > 1 + d(z, z^\prime)$. The low energy truncated propagator obeys  \begin{eqnarray}\label{short time low energy dispersive}\bigg|  \int_0^\infty    e^{i t \blambda^2} \chi_{\text{low}} d E_{P}(\blambda, z, z^\prime) \, d\blambda \bigg| \leq C |t|^{- 3/2} \big(1 + d(z, z^\prime)\big) e^{- n d(z, z^\prime)/2} \end{eqnarray} for all times. On the other hand,  we have short time dispersive estimates for the high energy truncated propagator microlocalized near the diagonal \begin{eqnarray}  \label{short time high energy dispersive} \begin{gathered} \bigg|  \int_0^\infty    e^{i t \blambda^2} \chi_\infty \Big(Q_k(\blambda) d E_{P}(\blambda) Q_k^\ast (\blambda)\Big) (z, z^\prime) \, d\blambda \bigg| \\ \quad\quad\quad \leq C |t|^{- (n + 1)/2}  (1 + d(z, z'))^{n/2} e^{- n d(z, z^\prime)/2} \end{gathered},\end{eqnarray} provided $t < 1 + d(z, z^\prime)$. \end{proposition}

\begin{remark}If we work on a manifold without conjugate points, this result will reduce to the dispersive estimates (\ref{hyperbolic dispersive}) on hyperbolic space, where the microlocalization is needless. Moreover, for short time estimates, say $t < 1 + d(z, z')$, we can combine \eqref{short time low energy dispersive} and \eqref{short time high energy dispersive}.\end{remark}

\begin{proof}[Proof of \eqref{long time dispersive microlocalized}]

Let us look at the long time dispersion first. Because we want to use stationary phase estimates, we have to split the amplitude of the microlocalized propagator into  functions compactly supported in $\blambda$. To do so, we select a bump function $\phi \in C_c^\infty[1/2, 2]$ such that $\sum_j \phi(2^{-j} \blambda) = 1$ and let $\phi_0(\blambda) = \sum_{j \leq 0} \phi(2^{-j} \blambda)$. Then the Schr\"{o}dinger propagator is decomposed as $I_0 + \sum_{j > 0} I_j$, which is  \begin{eqnarray*}I_0 & = & \int_0^\infty  e^{it\blambda^2} \chi_\infty(\blambda) Q_k(\blambda) dE_{P}(\blambda) Q_k^\ast(\blambda) \phi_0(\blambda) \, d\blambda\\ I_j &=& \int_0^\infty e^{it\blambda^2} \chi_\infty(\blambda) Q_k(\blambda) dE_{P}(\blambda) Q_k^\ast(\blambda)  \phi(2^{-j}\blambda) \, d\blambda. \end{eqnarray*}

$\bullet \ $ Case 1: $d(z, z') \leq 1$

As $t$ goes to infinity, the phase function is $\blambda^2$ which is clearly non-degenerate at the stationary point $\blambda = 0$.

Noting $0$ is not on the support of $\chi_\infty$, we have $I_0 = O(t^{-\infty})e^{-n d(z, z') / 2}$. On the other hand, noting the phase function of the $I_j$ terms are non-stationary, we deduce \begin{eqnarray*}\sum_{j > 0} |I_j| &=&  \sum_{j > 0}\bigg| \int_0^\infty \bigg(\frac{1}{2it\blambda}\frac{d}{d\blambda}\bigg)^N (e^{it\blambda^2}) dE_{P}(\blambda) \phi(2^{-j}\blambda) \, d\blambda \bigg|\\&\leq& C \sum_{j > 0} t^{-N}  e^{- nd(z, z')/2} \int_{2^{j - 1}}^{2^{j + 1}} \blambda^{n - 2N}\, d\blambda \leq C t^{- N} e^{- nd(z, z')/2}. \end{eqnarray*} Let $N$ go to $\infty$ to finish the proof of this case.

$\bullet \ $ Case 2: $d(z, z') \geq 1$

Since $d(z, z')$ goes to $\infty$ as well as $t$, the phase function consists of not only $\blambda^2$ but also some other term coming from the spectral measure. The outgoing and incoming parts of the spectral measure contribute the oscillatory terms $e^{- i \blambda d(z, z')}$ and $e^{ i \blambda d(z, z')}$ respectively. So the new phase function will be $t \blambda^2 \mp \blambda d(z, z')$. In the incoming case, such phase function isn't stationary. Then we can select the bump function $\phi$ as above and get compactly supported amplitudes. Noting the support of $\phi_0$ isn't intersected with $\chi_\infty$, we can obtain the dispersive estimates by running the same argument of non-stationary phase and integration by parts \begin{eqnarray*}\sum_{j > 0} |I_j| &=&  \sum_{j > 0} \bigg|\int_0^\infty \bigg(\frac{1}{2it\blambda + i d(z, z')}\frac{d}{d\blambda}\bigg)^N (e^{it\blambda^2 + i d(z, z') \blambda}) a_+(\blambda) \phi(2^{-j}\blambda) \, d\blambda\bigg| \\&\leq& C \sum_{j > 0} t^{-N} \int_{2^{j - 1}}^{2^{j + 1}} \blambda^{n  - 2N} e^{- nd(z, z')/2} \, d\blambda \leq C t^{- N} e^{- nd(z, z')/2}, \end{eqnarray*}for any large $N$.

 On the other hand, the phase function $t \blambda^2 - \blambda d(z, z')$ is stationary at $\blambda = d(z, z')/(2t)$. Nonetheless $d(z, z')/(2t) < 1$ doesn't lie on the support of $\chi_\infty$ either, we thus can prove the dispersive estimates by the same argument. The proof is now complete.

  \end{proof}

 \begin{proof}[Proof of \eqref{short time low energy dispersive}]

 It can be deduced from the results of the spectral measure at low energy. We make a change of variable and get
 \begin{eqnarray*}U_{\text{low}} &=& \int_0^\infty    e^{i t \blambda^2} \chi_{\text{low}}(\blambda) d E_{P}(\blambda, z, z^\prime) \, d\blambda  \\ & = &  t^{-1/2} \int_0^\infty    e^{i \blambda^2} \chi_{\text{low}}(t^{-1/2} \blambda) d E_{P}(t^{-1/2} \blambda, z, z^\prime) \, d\blambda.\end{eqnarray*} We decompose the LHS as $I_0 + I_\infty$, where\begin{eqnarray*}
I_0 &=& t^{-1/2} \int_0^1    e^{i \blambda^2} \chi_{\text{low}}(t^{-1/2} \blambda) d E_{P}(t^{-1/2} \blambda, z, z^\prime) \, d\blambda \\ I_\infty &=& t^{-1/2} \int_1^\infty    e^{i \blambda^2} \chi_{\text{low}}(t^{-1/2} \blambda) d E_{P}(t^{-1/2} \blambda, z, z^\prime) \, d\blambda.
\end{eqnarray*}

 We use \eqref{low spectral measure upper bound} for low energies to estimate $I_0$ as follows \begin{eqnarray*}
|I_0| &=& t^{-1/2}\bigg| \int_0^1    e^{i \blambda^2} \chi_{\text{low}}(t^{-1/2} \blambda) d E_{P}(t^{-1/2} \blambda, z, z^\prime) \, d\blambda \bigg|\\  &\leq& t^{-1/2} \int_0^1  (t^{-1/2} \blambda)^2 (1 + d(z, z')) e^{-n d(z, z')/2}\, d\blambda\\  &\leq& C t^{-3/2}  (1 + d(z, z')) e^{-n d(z, z')/2}.
\end{eqnarray*}

On the other hand,  we shall invoke \eqref{low energy spectral measure} for low energies and perform integration by parts on $I_\infty$.

$\bullet \ $ Case 1:  $t^{1/2} > 1 + d(z, z')$

 We perform integration by parts and deduce that
\begin{eqnarray*} I_\infty &=& t^{-1/2} \int_1^\infty  \bigg(\frac{1}{2i\blambda}\frac{d}{d \blambda}\bigg)^3  \big(e^{i \blambda^2}\big) \chi_{\text{low}}(t^{-1/2} \blambda) d E_{P}(t^{-1/2} \blambda, z, z^\prime) \, d\blambda\\&=& \frac{t^{-1}}{-8i} \int_1^\infty  \bigg(\frac{1}{\blambda}\frac{d}{d \blambda}\bigg)^3  \big(e^{i \blambda^2}\big) \chi_{\text{low}}(t^{-1/2} \blambda)\blambda (\rho_L\rho_R)^{n/2}\\ &&\times\bigg((\rho_L\rho_R)^{it^{-1/2}\blambda} a(t^{-1/2} \blambda) - (\rho_L\rho_R)^{- it^{-1/2}\blambda} a(- t^{-1/2} \blambda)\bigg)\, d\blambda\\
 &=& I_{\infty, 1} + I_{\infty, 2},\end{eqnarray*}
where we write \begin{eqnarray*}I_{\infty, 1} &=& \frac{t^{-1}(\rho_L\rho_R)^{n/2}}{8i}  \bigg(\frac{1}{\blambda}\frac{d}{d \blambda}\bigg)^2 \big(e^{i \blambda^2}\big) \\&& \times\bigg((\rho_L\rho_R)^{it^{-1/2}\blambda} a(t^{-1/2} \blambda) - (\rho_L\rho_R)^{- it^{-1/2}\blambda} a(- t^{-1/2} \blambda)\bigg)\bigg|_{\blambda=1}\\
 I_{\infty, 2} &=& \frac{t^{-3/2}(\rho_L\rho_R)^{n/2}}{8i} \int_1^\infty  \bigg( \frac{1}{\blambda}\frac{d}{d \blambda}\bigg)^2 \big(e^{i \blambda^2}\big)  \\&&\times\bigg(  (\rho_L\rho_R)^{ it^{-1/2}\blambda} a'( t^{-1/2} \blambda)+  (\rho_L\rho_R)^{- it^{-1/2}\blambda} a'(- t^{-1/2} \blambda)
 \\&&
+ i (\rho_L\rho_R)^{it^{-1/2}\blambda}\ln(\rho_L\rho_R) a(t^{-1/2} \blambda)   + i (\rho_L\rho_R)^{- it^{-1/2}\blambda}\ln(\rho_L\rho_R) a(- t^{-1/2} \blambda)
 \bigg) \, d\blambda ,\end{eqnarray*} with a smooth function $a$ supported on $[0, 1]$.
 We now use \eqref{asymptotic of distance function} to estimate $I_{\infty, 1}$. If $t < M$ provided $M$ is sufficiently large,
$$|I_{\infty, 1}| \leq t^{-1} e^{-nd(z, z')/2} \leq C M^{1/2} t^{-3/2}  e^{-nd(z, z')/2}.$$ On the other hand, if $t > M$ (i.e. $t^{-1/2}$ is very small), we then use the smoothness of $a$ at $0$ and obtain $$\bigg((\rho_L\rho_R)^{it^{-1/2}} a(t^{-1/2} ) - (\rho_L\rho_R)^{- it^{-1/2}} a(- t^{-1/2} )\bigg) \leq Ct^{-1/2}.$$ Consequently, we obtain that $$|I_{\infty, 1} | \leq C t^{-3/2}  e^{-nd(z, z')/2}.$$
For $I_{\infty, 2}$, by \eqref{asymptotic of distance function}, we observe the part of the integrand contained in the parentheses is bounded by $C (1 + d(z, z'))$.  We take integration by parts two more times and get \begin{eqnarray*}I_{\infty, 2} &\leq& C t^{-3/2} (1 + d(z, z')) e^{-nd(z, z')/2}\\&&\bigg(\int_1^\infty \blambda^{-4} \,d\blambda + t^{-1/2}(1 + d(z, z'))\int_1^\infty \blambda^{-3} \,d\blambda + t^{- 1} ( 1 + d(z, z')^2) \int_1^{\infty} \blambda^{-2}  \,d\blambda\bigg),\end{eqnarray*}  Noting $1 + d(z, z') \leq t^{1/2}$, we conclude that \begin{eqnarray*} I_{\infty, 2} \leq C t^{-3/2} (1 + d(z, z')) e^{-nd(z, z')/2}.\end{eqnarray*} 

$\bullet \ $ Case 2:  $1<  t^{1/2} < 1 + d(z, z')$, 

We shall estimates following integrals instead \begin{eqnarray*}I_{\infty, +} &=& t^{-1/2} e^{- n\tilde{d}}\int_1^\infty e^{i\blambda^2 + i\tilde{d} t^{-1/2}\blambda} \blambda a(t^{-1/2}\blambda) \, d\blambda\\ I_{\infty, -} &=& t^{-1/2}e^{- n\tilde{d}}\int_1^\infty e^{i\blambda^2 - i\tilde{d}t^{-1/2}\blambda} \blambda a(- t^{-1/2}\blambda)  \, d\blambda,\end{eqnarray*} where $\tilde{d} = \ln(\rho_L \rho_R)$ and $ a \in C_c^\infty [0,1]$. Without loss of generality, we assume $\tilde{d} \geq 0$. By \eqref{asymptotic of distance function}, $\tilde{d}$ is an approximation of the geodesic distance function. The term $I_{\infty, +}$ is easier, since the first derivative of the phase $2\blambda + \tilde{d}t^{-1/2}$ is not vanishing. So we directly adopt the standard integration by parts argument as follows. First, we insert an invariant operator \begin{eqnarray*}I_{\infty, +} \leq t^{-1} e^{- n\tilde{d}}\bigg|\int_1^\infty \frac{1}{2\blambda + \tilde{d} t^{-1/2}}\frac{\partial}{\partial \blambda} e^{i\blambda^2 + i\tilde{d} t^{-1/2}\blambda} \cdot \blambda a(t^{-1/2}\blambda) \, d\blambda \bigg|.\end{eqnarray*} Then we perform integration by parts on the integral and  get \begin{eqnarray*} \frac{a(t^{-1/2}) e^{i + i t^{-1/2} \tilde{d}}}{2 + t^{-1/2} \tilde{d}} + \int_1^\infty  e^{i\blambda^2 + i\tilde{d} t^{-1/2}\blambda}  \bigg(\frac{\tilde{d} t^{-1/2}a(t^{-1/2}\blambda)  }{(2\blambda + \tilde{d}t^{-1/2})^2} + \frac{\blambda t^{-1/2}a'(t^{-1/2}\blambda)  }{2\blambda + \tilde{d}t^{-1/2}}\bigg) \, d\blambda \end{eqnarray*} The boundary term and the first term of the integral is bounded by a constant, whilst the second term is yielded to $$C \int_1^{t^{1/2}} \frac{\blambda t^{-1/2}}{2\blambda + t^{-1/2} \tilde{d}} \,d\blambda.$$ Also noting $t^{-1/2}\tilde{d} > C$, we conclude that $$I_{\infty, +} \leq C t^{-3/2} e^{- n d(z, z')} (1 + d(z, z')).$$

We now estimate $I_{\infty, -}.$ Since $2\blambda - t^{-1/2} \tilde{d} \blambda$ might vanish, we have to take a dyadic decomposition. To do so, we introduce a partition of unity $\sum_j \phi(2^{-j} \blambda) = 1$ with $\phi \in C_c^\infty[1/2, 2]$. We further denote $$\psi_k(\blambda) = \phi\big( 2^{-k} |2\blambda - t^{-1/2} \tilde{d}(z, z')|\big).$$ Then we have to estimate following integrals \begin{eqnarray*} I_{\infty, -}^0 & = & t^{-1} e^{- n\tilde{d}} \int_1^\infty e^{i\blambda^2 - i t^{-1/2} \tilde{d} \blambda} \blambda a(t^{-1/2}\blambda) \sum_{k \leq 0} \psi_k(\blambda)  \,d\blambda,\\  I_{\infty, -}^k & = & t^{-1} e^{- n\tilde{d}} \int_1^\infty e^{i\blambda^2 - i t^{-1/2} \tilde{d} \blambda} \blambda a(t^{-1/2}\blambda)  \psi_k(\blambda)  \,d\blambda, \quad \quad\quad k > 0.\end{eqnarray*} We consider $I_{\infty, -}^0$ first. One can find a sufficiently large number $M$ such that for all $\blambda > M$ we have $\blambda \sim t^{-1/2} \tilde{d}$ if $|2\blambda - t^{-1/2} \tilde{d}| \leq 2$. Since the measure of the support of $\sum_{k \leq 0} \psi_k(\blambda)$ is smaller than $4$, we thus get $$\int_1^M e^{i\blambda^2 - i t^{-1/2} \tilde{d} \blambda} \blambda a(t^{-1/2}\blambda) \sum_{k \leq 0} \psi_k(\blambda)  \,d\blambda \leq C \leq C t^{-1/2} (1 + d(z, z')).$$ In the meantime, we have \begin{eqnarray*}\lefteqn{\int_M^\infty e^{i\blambda^2 - i t^{-1/2} \tilde{d} \blambda} \blambda a(t^{-1/2}\blambda) \sum_{k \leq 0} \psi_k(\blambda)  \,d\blambda}\\& \leq& t^{-1/2} \tilde{d}  \int_M^\infty e^{i\blambda^2 - i t^{-1/2} \tilde{d} \blambda} \frac{\blambda a(t^{-1/2}\blambda)}{t^{-1/2} \tilde{d}} \sum_{k \leq 0} \psi_k(\blambda)  \,d\blambda\\ &\leq& C t^{-1/2} (1 + d(z, z')).\end{eqnarray*} On the other hand, we again use integration by parts $N$ times on $I_{\infty, -}^k$ for $k > 0$. \begin{eqnarray*}\lefteqn{\sum_{k > 0}\bigg|\int_1^\infty e^{i\blambda^2 - i t^{-1/2} \tilde{d} \blambda} \blambda a(t^{-1/2}\blambda) \psi_k(\blambda) \, d\blambda\bigg|} \\&\leq& \sum_{k > 0} \bigg|\int_1^\infty \bigg(\frac{1}{2i\blambda - i t^{-1/2} \tilde{d}}\frac{\partial}{\partial \blambda}\bigg)^N e^{i\blambda^2 - i t^{-1/2} \tilde{d} \blambda} \blambda a(t^{-1/2}\blambda) \psi_k(\blambda) \, d\blambda\bigg|\\& \leq & C t^{-1/2} \tilde{d} \sum_{k > 0} 2^{-kN} \int_{|2\blambda - t^{-1/2} \tilde{d}| \sim 2^k} \blambda^{1 - N}\, d\blambda \\&\leq& C t^{-1/2} (1 + d(z, z')).\end{eqnarray*} Plugging  these estimates into $I_{\infty, -}^0$ and $I_{\infty, -}^k$ respectively, we conclude $$I_{\infty, -} \leq C t^{-3/2} (1 + d(z, z')) e^{-nd(z, z')/2}.$$ The proof is now complete.
\end{proof}

 \begin{proof}[Proof of \eqref{short time high energy dispersive}] Because of the distinction between the long and short distance, we discuss the two cases separately. In particular, the exponential decay is negligible in case of short distance,  as $e^{-n d(z, z')/2}$ is bounded from below.  The proof of \eqref{short time high energy dispersive} in case of small distance is the same with the proof of the dispersive estimates on asymptotically conic manifolds by Hassell and Zhang \cite{Hassell-Zhang}, as the spectral measure for small $d(z, z')$ and large $\blambda$ obeys the same estimates as on asymptotically conic manifolds. In fact, the idea for both long distance and short distance is to perform an appropriate dyadic decomposition over the value of the derivative of the phase function for an integration by parts argument. We only give the proof for long distance to see the more interesting exponential decay in $d(z, z')$.

First of all,  we rescale the microlocalized high energy truncated propagator $U_k$ as follows \begin{eqnarray*}U_k &=& \int_0^\infty e^{i t \blambda^2} \chi_\infty \Big(Q_k(\blambda) d E_{P}(\blambda) Q_k^\ast(\blambda)\Big) (z, z^\prime) \, d\blambda\\ &=&  t^{-1/2} \int_0^\infty e^{i  \blambda^2 } \chi_\infty(t^{-1/2} \blambda)  \Big(Q_k d E_{P} Q_k^\ast\Big) (t^{-1/2} \blambda, z, z^\prime)  \, d\blambda, \end{eqnarray*} provided $t < 1 + d(z, z')$. Applying Proposition \ref{derivative spectral measure estimate} for high energies, we write \begin{equation}\label{T+ and T-}U_k =t^{-(n + 1)/2}  T_{+} + t^{-(n + 1)/2} T_{-},\end{equation}  where\begin{eqnarray*} T_{+} &=&  \int_0^\infty e^{i  (\blambda^2 + t^{-1/2} \blambda d(z, z'))} \blambda^n a_+(t^{-1/2} \blambda, z, z^\prime)  \, d\blambda \\  T_{-} &=&   \int_0^\infty e^{i  (\blambda^2 - t^{-1/2} \blambda d(z, z')) } \blambda^n a_- (t^{-1/2} \blambda, z, z^\prime)   \, d\blambda \quad \end{eqnarray*} with smooth function $a_{\pm}(\blambda, z, z^\prime)$ on $[1, \infty) \times X^2_0$ obeying  \begin{eqnarray*}  \bigg|\frac{d^j}{d\blambda^j} a_{\pm}(t^{-1/2} \blambda, z, z^\prime) \bigg|=  O\Big(t^{n/4}\blambda^{- n/2 - j} e^{- n d(z, z^\prime) / 2}\Big) & \mbox{if $d(z, z^\prime)$ is large}.\end{eqnarray*} Now it suffices to prove both $T_+$ and $T_-$ are bounded by $(1 + d(z, z'))^{n/2} e^{- nd(z, z')/2}.$

We decompose the $T_+$ term further into $\sum_{j \geq 0} T_{j, +}$, where
\begin{eqnarray*} T_{j, +}& = &  \int_0^\infty e^{i  (\blambda^2 + t^{-1/2} \blambda d(z, z'))} \blambda^n a_+(t^{-1/2} \blambda, z, z^\prime) \phi(2^{- j}  \blambda) \, d\blambda \quad \mbox{for $j > 0$} \\ T_{0, +} &=  & \int_0^\infty e^{i  (\blambda^2 + t^{-1/2} \blambda d(z, z'))} \blambda^n a_+(t^{-1/2} \blambda, z, z^\prime) (1 - \sum_{j > 0}\phi(2^{- j}  \blambda)) \, d\blambda ,\end{eqnarray*} where  we denote, by a partition of unity $\sum_j \phi(2^{- j} \blambda) = 1$ with $\phi \in C_c^\infty [1/2, 2]$. It is clear that $T_{0, +}$ is bounded by $(1 + d(z, z'))^{n/2} e^{- nd(z, z')/2}$ after a quick application of Proposition \ref{derivative spectral measure estimate}.
On the other hand, for each $T_{j, +}$, the phase function of this oscillatory integral is actually non-stationary. One thus can insert a  differential operator $N$ times leaving the exponential term invariant and take integration by parts \begin{eqnarray*} |T_{j, +}| &=& \bigg|\int_0^\infty \bigg(\frac{1}{i(2\blambda + t^{-1/2} d(z, z^\prime))} \frac{\partial}{\partial \blambda} \bigg)^N e^{i(\blambda^2 + t^{-1/2} \blambda d(z, z^\prime))}   \blambda^n a_+(t^{-1/2} \blambda, z, z^\prime) \phi(2^{- j} \blambda)\, d\blambda\bigg|\\ &\leq&  C \int_{|\blambda| \sim 2^j} t^{n/4 } e^{-n d(z, z')/2}\blambda^{n/2 - 2N} \, d\blambda \\ &\leq&  C (1 + d(z, z'))^{n/2} e^{- n d(z, z')/2} \int_{|\blambda| \sim 2^j} \blambda^{n/2 - 2N} \, d\blambda .\end{eqnarray*} The sum of $T_{j, +}$s  in $j$ is clearly convergent if we make $N$ sufficiently large.

For the term $T_-$, the phase function may be stationary, we have to make a subtler decomposition. One may rewrite the integral as $ T_- = \sum_{k \geq 0} T_{k, -},$where \begin{eqnarray*}T_{0, -} &=&  \int_0^\infty e^{i \blambda^2 - i  t^{-1/2} \blambda d(z, z')}  \blambda^n
a_-(t^{-1/2} \blambda, z, z^\prime) \sum_{k \leq 0} \psi_k(\blambda)  \, d\blambda\\ T_{k, -} &=&  \int_0^\infty e^{i \blambda^2 - i t^{-1/2} \blambda d(z, z')}  \blambda^n a_-(t^{-1/2} \blambda, z, z^\prime) \psi_k(\blambda)  \, d\blambda, \quad \mbox{$k > 0$} \\ \psi_k(\blambda) &=& \phi\Big(2^{- k} \big|2 \blambda - t^{-1/2} d(z, z^\prime)\big|\Big) .\end{eqnarray*}

If we plug the estimates for $a_-$ in $T_{0, -}$, we will have $T_{0, -}$ bounded by $$(1 + d(z, z'))^{n/2}  e^{- n d(z, z^\prime)/2}  \int_0^\infty  \frac{t^{n/4}}{(1 + d(z, z'))^{n/2}\blambda^{n/2} } \blambda^n \sum_{k \leq 0} \psi_k(\blambda)   \, d\blambda.$$ The latter integral is convergent. In fact, if $t^{-1/2} d(z, z^\prime) $ is bounded, $\blambda$ will also be bounded, because of $$\text{supp} \, \bigg(\sum_{k \leq 0}\psi_k\bigg) = \{|2\blambda - t^{-1/2} d(z, z^\prime)| \leq 2\}.$$ Therefore, the conditions of large distance (large $d(z, z')$), short time (small $t$), and high energy (large $\blambda$), make the fraction in the integrand also bounded on the domain. So the $\blambda$-integration is convergent. If $t^{-1/2} d(z, z^\prime)$ is large, the restriction $|2\blambda - t^{-1/2} d(z, z^\prime)| \leq 2$ from the support of $\sum_{k \leq 0} \psi_k$ implies $\blambda \sim t^{-1/2} d(z, z^\prime)$.  Consequently,  for any value of $t^{-1/2} d(z, z')$, we have $$\frac{t^{n/4}}{(1 + d(z, z'))^{n/2}\blambda^{n/2} } \blambda^n \leq C $$  Then the integral $T_{0, -}$ is bounded by $$C \int_{\{2\blambda - t^{-1/2} d(z, z^\prime)  < 2\}}\sum_{k < 0}\psi_k \, d\blambda \leq C.$$

For the $T_{k, -}$ terms, since $|2\blambda - t^{-1/2} d(z, z^\prime)| > 1$, namely the phase function is non-stationary, we can employ the integration by parts argument. We denote $$ L_- = \frac{1}{2\blambda - t^{-1/2} d(z, z^\prime)} \frac{d}{d\blambda}$$ and get following estimates for $\sum_{k > 0} T_{k, -}$ \begin{eqnarray*} \lefteqn{ \sum_{k > 0} \bigg|\int_0^\infty  e^{i (\blambda^2 -  t^{-1/2} d(z, z^\prime)\blambda )} \blambda^n  a_-\big(t^{-1/2} \blambda, z, z^\prime\big)  \psi_k(\blambda)   \, d\blambda \bigg|}\\ &=&   \sum_{k > 0} \bigg|\int_0^\infty L_-^N \Big(e^{i (\blambda^2 -  t^{-1/2} d(z, z^\prime)\blambda )}\Big) \blambda^n a_-\big(t^{-1/2} \blambda, z, z^\prime\big)  \psi_k(\blambda) \, d\blambda \bigg| \\ &\leq& C   e^{- n d(z, z^\prime)/2}\sum_{k > 0}  2^{- k N} \int_{|2\blambda - t^{-1/2}d(z, z^\prime)| \sim 2^k} \blambda^{n/2 - N} t^{n/4} \, d\blambda \\ &\leq& C  (1 + d(z, z'))^{n/2} e^{- n d(z, z^\prime)/2},\end{eqnarray*} provided $N$ is large enough.

\end{proof}

\begin{remark}A key point of this integration by parts argument in this proof is that we can have a non-stationary phase function in the oscillatory integral. To get the dispersive estimates for the propagator $U_i(t)U_j^\ast(s)$, we will get a non-stationary phase function and run this argument again. \end{remark}

\section{Dispersive estimates for Schr\"{o}dinger propagators II}\label{sec : dispersive2}

Because we will have to establish the retarded estimates for the Strichartz estimates, the off-diagonal microlocalized spectral measure $Q_j dE_{P} Q_k^\ast$ will confront us. Before stating off-diagonal microlocalized dispersive estimates, we have to define all relations of the microlocalization pairs $(Q_j, Q_k^\ast)$.

This was discussed by Guillarmou and Hassell \cite{Guillarmou-Hassell} for Sobolev estimates, which is closely related to Strichartz estimates. Let us review their notions of outgoing / incoming relations. Suppose $g^t$ is the geodesic/bicharacteristic flow and $Q, Q^\prime$ are two semiclassical pseudodifferential operators of semiclassical order $0$ and differential order $- \infty$. We say $Q$ is not outgoing related to $Q^\prime$ if the forward flowout $g^t(\text{WF}_h(Q^\prime))$ with $t \geq 0$ doesn't meet $\text{WF}_h(Q)$, whilst $Q$ is not incoming related to $Q^\prime$ if the backward flowout $g^t(\text{WF}_h(Q^\prime))$ with $t \leq 0$ doesn't meet $\text{WF}_h(Q)$. It is useful to note $Q$ not incoming related to $Q^\prime$ is equivalent to $Q^\prime$ not outgoing related to $Q^\prime$.

\begin{proposition}[Dispersive estimates II]\label{dispersive estimate2}There is a refined pseudodifferential operator partition of unity $Id = \sum_{k = 0}^N Q_k$ such that \begin{eqnarray} \label{dispersive2 long} \begin{gathered}\bigg| \int_0^\infty e^{i (t - s) \blambda^2} \chi_\infty \Big(Q_j(\blambda) d E_{P}(\blambda) Q_k^\ast (\blambda)\Big) (z, z^\prime) \, d\blambda \bigg|\\  \quad  \leq C |t - s|^{- \infty}  e^{- n d(z, z^\prime)/2}\end{gathered} && \quad \mbox{for $|t - s| > 1 + d(z, z')$}\\ \label{dispersive2 short} \begin{gathered}\bigg| \int_0^\infty e^{i (t - s) \blambda^2} \chi_\infty \Big(Q_j(\blambda) d E_{P}(\blambda) Q_k^\ast (\blambda)\Big) (z, z^\prime) \, d\blambda \bigg| \\  \quad  \leq C |t - s|^{- (n + 1)/2} (1 + d(z, z'))^{n/2} e^{- n d(z, z^\prime)/2}\end{gathered}&&\quad \mbox{for $|t - s| < 1 + d(z, z')$},\end{eqnarray} hold for all $t \neq s $ if $\text{WF}_h(Q_j) \cap \text{WF}_h(Q_k) \neq \emptyset$, for $t < s$ if $Q_j$ is not outgoing related to $Q_k$, for $s < t$ if $Q_j$ is not incoming related to $Q_k$. \end{proposition}

Before proving the dispersive estimates, we have to refine the microlocalization for spectral measure and categorize all microlocalization pairs $\{(Q_j, Q_k^\ast)\}_{j, k = 1}^N$, where $Q_0$ is neglected as it is not on the semiclassical wavefront set of spectral measure. \begin{lemma}\label{classification of microlocalization} The microlocalization pair $(Q_j, Q_k^\ast)$ with $j, k\geq 1$ must obey one of the following relations:

\begin{enumerate}\renewcommand{\labelenumi}{$($\roman{enumi}$)$}

\item $Q_j$ is not outgoing related to $Q_k$

\item $Q_j$ is not incoming related to $Q_k$

\item  The off-diagonal microlocalized spectral measure at high energy takes the form \begin{eqnarray}\label{spectral measure estimates}  Q_j(\blambda) \frac{d^j}{d\blambda^j} dE_{P}(\blambda) Q_k^\ast (\blambda) = e^{i \blambda d(z, z')} \blambda^n M_+ + e^{-i \blambda d(z, z')} \blambda^n M_- ,\quad \mbox{for large $\blambda$,}\end{eqnarray} where $M_\pm$ are defined on the forward and backward bicharacteristic flow respectively and satisfying \begin{eqnarray*}\frac{d^j M_\pm}{d\blambda^j} = \left\{ \begin{array}{l@{\quad , \quad}l}  O\Big(\blambda^{- j}\big(1 + \blambda d(z, z^\prime)\big)^{- n/2} \Big) & \mbox{if $d(z, z^\prime)$ is small}\\ O\Big(\blambda^{- n/2 - j} e^{- n d(z, z^\prime) / 2}\Big) & \mbox{if $d(z, z^\prime)$ is large}\end{array}\right..\end{eqnarray*}
\end{enumerate}\end{lemma}

\begin{proof}\footnote{The proof is essentially  due to Guillarmou and Hassell \cite{Guillarmou-Hassell}.}

Recall in Section \ref{sec : spectral measure}  we have taken a negatively curved strip neighbourhood of $\partial X$, say $\{x \leq 2\epsilon\}$.  Also recall the cubes $Q_1, \dots, Q_{N_1}$ supported on $B_1, \dots, B_{N_1} $, affiliated with $I_1, \cdots, I_{N_1}$, over $\{x \leq \epsilon\}$ and $Q_{N_1 + 1}, \dots, Q_{N_2}$ supported on $B_{N_1 + 1}, \dots, B_{N_2}$ over $\{x > \epsilon\}$. Now we also assume $\{x \geq 2\epsilon\}$ is compact and geodesically convex. It is true as long as $\epsilon$ is sufficiently small.

Therefore there are following cases of microlocalization:

\begin{enumerate}

\item $(Q_j, Q_k^\ast)$ with $B_j \cap B_k \neq \emptyset$, $1 \leq j \leq N_1$ and $1 \leq k \leq N_1$

Since $B_j$ is intersected with $B_k$, then $I_j$ is intersected with $I_k$. One may prescribe $k = j + 1$. Note both $I_j$ and $I_{j + 1}$ are small subintervals in $[-3/2, 3/2]$ and they are contained in  $I_{j} \cup I_{j + 1}$. Then one can find a slice $B_{j, j + 1}$ in $\{x < \epsilon\} \cap \{\lambda \in I_j \cup I_{j + 1}\}$. Since $I_{j} \cup I_{j + 1}$ is a small interval in $[-3/2, 3/2]$ too,  we can find a pseudodifferential operator $Q_{j, j+1}$ microlocally supported on $B_{j, j+1}$ such that  $Q_{j, j+1} dE_{P} Q_{j, j+1}^\ast $ satisfies \eqref{spectral measure estimates}. So does $Q_{j} dE_{P} Q_{k}^\ast$.

\item $(Q_j, Q_k^\ast)$ with $B_j \cap B_k \neq \emptyset$, $N_1 \leq j \leq N_2$ and $N_1 \leq k \leq N_2$ 

Since the diameter of $B_j \cup B_k$ is bounded by a very small number, $B_j$ and $B_k$ are contained in a very small cube $B_{jk}$. Then $Q_{jk} dE_{P} Q_{jk}^\ast$, with $Q_{jk}$ microlocally supported on $B_{jk}$, satisfies (\ref{spectral measure estimates}). Consequently, so does $Q_{j} dE_{P} Q_{k}^\ast$.

\item $(Q_j, Q_k^\ast)$ with $B_j \cap B_k \neq \emptyset$, $1 \leq j \leq N_1$ and $N_1 \leq k \leq N_2$ 

Since the diameter of $B_k$ is very small in the sense of Sasaki distance, we can narrow the range of $\lambda$ variable in $I_j$ and the range of $x$ variable in $\{x < 2 \epsilon\}$ such that both $B_k$ and $B_j$ are contained in a small slice $B_{jk}$ near the boundary with $\{\lambda \in I_j\}$. Then we again find a pseudodifferential operator $Q_{jk}$ microlocally supported on $B_{jk}$ such that  $Q_{jk} dE_{P} Q_{jk}^\ast$ satisfies (\ref{spectral measure estimates}).

\item $(Q_j, Q_k^\ast)$ with $B_j \cap B_k = \emptyset$, $1 \leq j \leq N_1$ and $1 \leq k \leq N_1$

Recall from \cite{Melrose-Sa Barreto-Vasy} or \cite{Chen-Hassell1} that the $0$-Hamilton vector field with Hamiltonian $p = \lambda^2 + h(y, \lambda, \mu)$ on asymptotically hyperbolic manifold $X$ is $$x \frac{\partial p}{\partial \lambda} \frac{\partial}{\partial x} + x \frac{\partial p}{\partial \mu} \cdot \frac{\partial}{\partial y} - \bigg(\mu \cdot \frac{\partial p}{\partial \mu} + x \frac{\partial p}{\partial x}\bigg) \frac{\partial}{\partial \lambda} + \bigg(\frac{\partial p}{\partial \lambda} \mu - x \frac{\partial p}{\partial y}\bigg) \cdot \frac{\partial}{\partial \mu}.$$ The variable $\lambda$, along the geodesic, decreases down to $-1$ , in a small neighbourhood of the boundary.

Without loss of generality, one may assume $\inf (I_j) > \sup (I_k)$. Take a geodesic $\gamma(t)$ with $\gamma(0) \in B_k$. If $\gamma(t)$ stays in $\{x < \epsilon \}$ for $t \geq 0$,  $\{\gamma(t) : t \geq 0\}$ will be disjoint from $B_j$, since $\lambda$ is nonincreasing along the forward bicharacteristic near the boundary. On the other hand, if $\gamma(t)$ goes beyond $\{x < \epsilon \}$ at time $t_2$,(i.e. $\gamma(t_2) \in \{x \geq \epsilon\}$) we have $\lambda(0) > 0$, hence $\inf (I_j) > \sup (I_k) > 0$. So we can find a maximal interval $(t_1, t_3)$ containing $t_2$ on $\{x \geq \epsilon\}$ such that $\lambda(t) > 0$ for all $t < t_1$ and  $\lambda(t) < 0$ for all $t > t_3$, since $\lambda$ is nonincreasing in $\{x < \epsilon\}$.  Consequently, $\gamma$ is disjoint from $B_j$ whenever $t > 0$: when $0 < t < t_1$ i.e. $\lambda < \lambda(0) < \inf (I_j)$; when $t_2 < t < \infty$ i .e. $\lambda < 0 < \inf (I_j)$.

\item $(Q_j, Q_k^\ast)$ with $B_j \cap B_k = \emptyset$, $1 \leq j \leq N_1$ and $N_1 \leq k \leq N_2$

Take a geodesic $\gamma$ with $\gamma(0) \in B_j$. If $\sup I_j < 0$, then $x(t)$ is non-increasing namely $\gamma(t)$ will stay in $\{x < \epsilon\}$ for $t > 0$ and be disjoint from $B_k$. In the meantime, if $\inf I_j > 0$, $\gamma(t)$ will stay in $\{x < \epsilon\}$ for $t < 0$. If $0 \in I_j$ and $\lambda(t_0) = 0$, $x(t)$ is nonincreasing for all $t > t_0$, since $\lambda$ is non-positive afterwards. So $\gamma(t)$ will stay in $\{x < \epsilon\}$ for all $t > t_0$.

\item $(Q_j, Q_k^\ast)$ with $B_j \cap B_k = \emptyset$, $N_1 \leq j \leq N_2$ and $N_1 \leq k \leq N_2$

Consider the function $(z, t) \rightarrow x(g^t)$. Since $d x(g^t(z)) / dt \neq 0$ locally in $\{x > \epsilon/2\}$, we apply implicit function theorem and get an implicit function $t(z)$. We can find a time $t(z)$ such that $x(g^{t(z)}) = \epsilon/2$. Therefore, for any compact set $K \subset \{|\zeta| = 1\} \cap \{x > \epsilon/2\}$,  there is a $T_+ > 0$ respectively $T_- < 0$ such that $g^t(K) \subset \{x < \epsilon / 2\}$ for all $t > T_+$ respectively $t < T_-$. Assuming $B_j$ and $B_k$ are outgoing related and incoming related, we shall get a contradiction. Under this hypothesis and the compactness, there exist two sequences of points $\{z_l\} \subset \{x > \epsilon/2\} $ and  $\{z_l^\prime\} \subset \{x > \epsilon/2\}$ with two sequences of times $\{t_l : t_l < - \iota < 0\}$ and $\{t_l^\prime : t_l^\prime > \iota > 0\}$ both going to the same point $z \in \{z > \epsilon/2\}$ via the the geodesic $g^t$, that is $$\lim_{l \rightarrow \infty} g^{t_l} (z_l) = \lim_{l \rightarrow \infty} g^{t^\prime_l} (z_l^\prime).$$  Since $T_- \leq t < - \iota < \iota < t^\prime \leq T_+$, we can find accumulation points $t$ and $t^\prime$ respectively. Then we have $g^t(z) = g^{t^\prime}(z)$. It gives a periodic geodesic which contradicts the non-trapping condition. Therefore we have either   $Q_j$ is not outgoing related to $Q_k$ or $Q_j$ is not incoming related to $Q_k$.

\end{enumerate}

\end{proof}

With this lemma, we can prove the dispersive estimates for off-diagonal microlocalized high energy truncated Schr\"{o}dinger propagators.

\begin{proof}[Proof of Proposition \ref{dispersive estimate2}] It is proved by the argument of Proposition \ref{dispersive estimate1} with minor changes based on the classification of microlocalizations in Lemma \ref{classification of microlocalization}.

 If $Q_j dE_{P}(\blambda) Q_k^\ast$ obeys (\ref{spectral measure estimates}), we can get desired dispersive estimates by repeating the proof of Proposition \ref{dispersive estimate1}.

  If $Q_j$ and $Q_k$ are not outgoing related\footnote{The proof is exactly the same in case $Q_j$ and $Q_k$ are not incoming related.}, we claim the $U_j(t) U_k^\ast(s)$ for $t < s$ is a Fourier integral operator $$\int_0^\infty \int_{\mathbb{R}^{N}} e^{ i (t - s) \blambda^2 + \blambda \phi} a(h, z, z', \theta) \, d\theta d\blambda,$$ provided $\phi(z, z', \theta) < - \epsilon < 0$ is the phase function of the wavefront set $\Lambda$ of the spectral measure. Since we can always write the propagator in this integral form, the only point we need to justify is that $\phi(z, z', \theta) < - \epsilon < 0$.

Recall that the forward bicharacteristic flow-out is that the flow-out of the Hamilton vector field of the metric function. By standard theory of Lagrangian distributions, the phase function $\phi$ can parametrize the forward flow-out in the following way that is $\Lambda^+$ is locally furnished coordinates $$\{(z, \phi^\prime_z) | \phi_\theta' = 0\}.$$ Hence phase function $\phi$ of forward bicharacteristic flow-out $\Lambda^+$ satisfies $$\phi(z, z', \theta) = r(z, z') \geq d(z, z'), \quad \mbox{when $\phi'_\theta = 0$}$$ where $r$ is the curve length along the bicharacteristic and $d$ is the geodesic distance. Since $Q_j$ is not outgoing related to $Q_k$, i.e. the forward geodesic flow-out of $\text{WF}_h Q_k^\ast$ doesn't meet $\text{WF}_h Q_j$, they are connected by the backward flow-out, namely $$\phi = -r(z, z') \leq -d(z, z') < 0.$$ The not outgoing relation gives a constantly negative sign of the phase function $\phi$ of microlocalized spectral measure $Q_j(\blambda) dE_{P}(\blambda) Q_k^\ast(\blambda)$.   Since $t - s < 0$, the phase function of the propagator is negative. So it allows us to play the integration by parts argument in Proposition \ref{dispersive estimate1} by the differential operator $$\frac{-i}{2 \blambda -  \phi / \sqrt{s - t}} \frac{\partial}{\partial \blambda}$$ to get the prove \eqref{dispersive2 short} , instead of $-i/(2 \blambda -  t^{-1/2} d(z, z')) \partial_\blambda$ in the proof of \eqref{short time high energy dispersive}. On the other hand, noting $\phi$ and $t-s$ have the same sign, namely the phase is non-stationary, we apply the rapid decay estimates, which reaily shows \eqref{dispersive2 long}.

\end{proof}

\section{Strichartz estimates}

We turn to proving Theorem \ref{inhomogeneous strichartz}. 

First of all, we shall establish the Strichartz estimates $$\|u\|_{L^q(\mathbb{R}, L^r(X))} \leq C \|f\|_{L^2(X)}$$for the homogeneous equations (i. e. $F\equiv 0$). Recall the low energy truncated propagator  and high energy microlocalized propagators. The solution $u$ of the homogeneous equation reads $$e^{itn^2/4}u(t, x) = \bigg( U_{\text{low}}(t) + \sum_{j = 0}^N U_j(t) \bigg) f (z)$$ for $0 < t < 1$, where $$U_{\text{low}}(t) = \int_0^\infty e^{it\blambda^2} \chi_{\text{low}}(\blambda) \, dE_{P}(\blambda) \quad \mbox{and} \quad  U_j = \int_0^\infty e^{it\blambda^2} \chi_\infty(\blambda) Q_j(\blambda) \, dE_{P}(\blambda).$$

The Strichartz estimates for homogeneous equations  $$\|e^{i t P^2} f(z)\|_{L^q_t L^r_z} \leq C \|f\|_{L^2_z}$$ are equivalent to $$\bigg\|\int e^{- i s P^2} G(s, z) \, ds \bigg\|_{L^2_z} \leq C \|G\|_{L^q_t L^r_z}.$$ Noting the decomposition $$e^{i(-s)P^2} = U_{\text{low}}^\ast(s) + \sum_{j = 0}^N U_j^\ast(s),$$   it suffices to show $$\bigg\|\int U_k^\ast(s) G(s, z) \, ds \bigg\|_{L^2_z} \leq C \|G\|_{L^q_t L^r_z},$$ where $k \in \{0 , 1, \dots, N, \text{low}\}$. By $TT^\ast$, it is equivalent to $$\bigg\|\int \big( U_k(t)U_k^\ast(s) F(s, z) \big) ds \bigg\|_{L_t^{q} L_z^{r}} \leq C \|F\|_{L_s^{q^\prime} L_z^{r^\prime}} .$$

One can split the left hand side by time. The long time part reduces to \begin{eqnarray}\label{short time homogeneous Strichartz norm}\bigg(\int  \bigg\| \int_{|t - s| \geq 1 }  U_k(t) U_k^\ast(s) F(s, z) \, ds\bigg\|_{L^r_z}^q\, dt \bigg)^{1/q} , \end{eqnarray} in the meantime, the short time part reduces to \begin{eqnarray}\label{long time homogeneous Strichartz norm}\bigg(\int  \bigg\| \int_{|t - s| \leq 1 } U_k(t) U_k^\ast(s) F(s, z)   \, ds\bigg\|_{L^r_z}^q  \, dt \bigg)^{1/q} . \end{eqnarray} 

To estimate these integrals, we need following mapping properties of the propagators, which we shall prove in the last section,

\begin{lemma}[Long times]\label{long time lemma}Suppose $|t - s| \geq 1$ and $2< r, \tilde{r} \leq \infty$. Then the following inequalities hold \begin{eqnarray} \label{dispersive low long} &\| U_{\text{low}}(t)U_{\text{low}}^\ast(s)\|_{L_z^{\tilde{r}^\prime} \rightarrow L_z^r}  \leq C  |t - s|^{-3/2} &  \\ \label{dispersive high long}&\| U_j(t)U_k^\ast(s)\|_{L_z^{\tilde{r}^\prime} \rightarrow L_z^r}  \leq C  |t - s|^{-3/2} &  \end{eqnarray}where the last one only holds for either $t - s > 1$ or $s - t > 1$ if $j \neq k$ and for both if $j = k$.\end{lemma}

\begin{lemma}[Short times]\label{short time lemma}Suppose $0 < |t - s| < 1$ and $2< r, \tilde{r} \leq \infty$. Then the following inequalities hold\begin{eqnarray} \label{dispersive low} &\| U_{\text{low}}(t)U_{\text{low}}^\ast(s)\|_{L_z^{\tilde{r}^\prime} \rightarrow L_z^r}  \leq C  |t - s|^{- \max\{1/2 - 1/r, 1/2 - 1/\tilde{r}\} (n + 1)} & \\ \label{dispersive high}&\| U_j(t)U_k^\ast(s)\|_{L_z^{\tilde{r}^\prime} \rightarrow L_z^r}   \leq C |t - s|^{- \max \{1/2 - 1/r, 1/2 - 1/\tilde{r}\} (n + 1)} &  ,\end{eqnarray} where the last one only holds for either $0 < t - s< 1$ or $0 < s - t < 1$ if $j \neq k$ and for both if $j = k$.\end{lemma}

Assuming these lemmas for the moment, we now continue the proof of Strichartz estimates.

We insert \eqref{dispersive high long} and \eqref{dispersive low long}  into \eqref{short time homogeneous Strichartz norm} and get \begin{eqnarray*}\lefteqn{\bigg(\int\bigg(\int_{|t - s| \geq 1 } \Big\| U_k(t) U_k^\ast(s) F(s, z) \Big\|_{L^r_z} \, ds \bigg)^q\,dt \bigg)^{1/q}}\\ &\leq& C \bigg(\int \bigg(\int_{|t - s| \geq 1 } |t - s|^{-3/2} \|  F(s, z) \|_{L^{r'}_z} \, ds \bigg)^q\, dt \bigg)^{1/q} \\ &\leq & \|F(s, z)\|_{L^{q'}_s L_z^{r'}}. \end{eqnarray*} We remark the kernel $|t - s|^{-3/2}\chi_{|t -s| > 1}$ is integrable so it maps $L^{q'}(\mathbb{R})$ to $L^q(\mathbb{R})$ for any $q \geq 2$, where no admissibility is needed.

On the other hand, one can use the short time  estimates (\ref{dispersive low}) and (\ref{dispersive high}). For $(q, r) \neq (2, 2(n + 1)/(n - 1))$, we invoke the admissibility condition \eqref{schrodinger admissibility} and Hardy-Littlewood-Sobolev inequality  \begin{eqnarray*}\lefteqn{\bigg(\int\bigg( \int_{|t - s| \leq 1 }  \| U_k(t) U_k^\ast(s) F(s, z) \|_{L^r_z}  \, ds \bigg)^q\,dt \bigg)^{1/q} }\\ & \leq & C  \bigg(\int\bigg(\int_{|t - s| \leq 1 }  \frac{1}{|t - s|^{ (1/2 - 1/r) (n + 1)}}  \|  F(s, z)  \|_{L^{r^\prime}_z} \, ds\bigg)^q\, dt\bigg)^{1/q}\\ & \leq & C  \bigg(\int\bigg(\int_{|t - s| \leq 1 }  \frac{1}{|t - s|^{ 2/q}}  \|  F(s, z)  \|_{L^{r^\prime}_z} \, ds\bigg)^q\, dt\bigg)^{1/q} \\& \leq &  C \|  F(s, z)  \|_{L^{q^\prime}_s L^{r^\prime}_z}. \end{eqnarray*} Here the last inequality requires $q < 2$, which is invalid for endpoints.

The short time endpoint estimates are proved via dispersive estimates and energy estimates by the standard Keel-Tao argument.

Next, following an argument from \cite{Hassell-Zhang}, we prove the inhomogeneous Strichartz estimates with the homogeneous estimates we have proved, that is $$\|e^{itP^2} f(z)\|_{L^q_t L^r_z} \leq C \|f\|_{L^2_z},$$ provided $(q, r)$ satisfies \eqref{schrodinger admissibility}. These estimates are equivalent to $$\bigg\|\int e^{i(t-s)P^2} F(s)\bigg\|_{L^q_sL^r_z} \leq C \|F\|_{L_t^{\tilde{q}'}L_z^{\tilde{r}'}},$$ provided $(\tilde{q}, \tilde{r})$ also satisfies \eqref{schrodinger admissibility}. By Duhamel's formula, the desired inhomogeneous Strichartz estimates are equivalent to the retarded estimates \begin{equation}\label{retarded estimates}\bigg\|\int_{s < t} e^{i(t-s)P^2} F(s)\bigg\|_{L^q_sL^r_z} \leq C \|F\|_{L_t^{\tilde{q}'}L_z^{\tilde{r}'}}.\end{equation} For the non-endpoint case i.e. neither of $(q, r)$ and $(\tilde{q}, \tilde{r})$ is $(2, 2(n + 1)/(n - 1))$, the retarded estimates \eqref{retarded estimates} follow immediately from Christ-Kiselev lemma \begin{lemma}[\cite{Christ-Kiselev}]Let $X, Y$ be Banach spaces, let $I$ be a time interval, let $K \in C^0(I \times I)$ be a kernel taking values in the space bounded operators from $X$ to $Y$. Suppose that $1 \leq p < q \leq \infty$ and $$\bigg\|\int_I K(t, s) f(s)\,ds \bigg\|_{L^q_t(I \rightarrow Y)} \leq C \|f\|_{L_t^p(I \rightarrow X)}.$$ Then one has $$\bigg\|\int_{\{s \in I : s < t\}} K(t, s)f(s) \,ds\bigg\|_{L_t^q(I \rightarrow Y)} \leq C \|f\|_{L_t^p(I \rightarrow X)}.$$\end{lemma} 

On the other hand, in order to establish the endpoint inhomogeneous estimates, we end up with following bilinear estimates $$\int\!\!\!\int_{s < t} \langle e^{i(t - s) P^2} F(s), G(t) \rangle \, dsdt \leq C \|F\|_{L^2_t L_z^{r'}} \|G\|_{L^2_tL_z^{r'}}.$$ Plugging in the decomposition of the propagator, we have to establish following estimates \begin{eqnarray}\label{retarded jk}\int\!\!\!\int_{s < t} \langle U_j(t) U_k^\ast(s) F(s), G(t) \rangle \, dsdt &\leq& C \|F\|_{L^2_t L_z^{r'}} \|G\|_{L^2_tL_z^{r'}}\\ \label{retarded lowlow}\int\!\!\!\int_{s < t} \langle U_{\text{low}}(t) U_{\text{low}}^\ast(s) F(s), G(t) \rangle \, dsdt &\leq& C \|F\|_{L^2_t L_z^{r'}} \|G\|_{L^2_tL_z^{r'}}\\ \label{retarded jlow} \int\!\!\!\int_{s < t} \langle U_\text{low}(t) U_k^\ast(s) F(s), G(t) \rangle \, dsdt &\leq& C \|F\|_{L^2_t L_z^{r'}} \|G\|_{L^2_tL_z^{r'}}.\end{eqnarray}

We want to use the standard Keel-Tao endpoint argument \cite[Section 7]{Keel-Tao}. So we have to establish the dispersive estimates and energy estimates for these propagators. The energy estimates are proved in Proposition \ref{energy estimate}. The dispersive estimates for $U_{\text{low}}(t) U_{\text{low}}^\ast(s)$ are proved in \eqref{short time low energy dispersive}, while $U_\text{low}(t) U_k^\ast(s)$ actually satisfies \eqref{short time low energy dispersive} as well. Therefore \eqref{retarded lowlow} and \eqref{retarded jlow} are proved by the standard Keel-Tao retarded estimates. 

The tricky one is \eqref{retarded jk}. According to Proposition \ref{dispersive estimate2}, we only have the dispersive estimates for $U_j(t) U_k^\ast(s) $ when  $t < s$ in the case of $Q_j$ is not outgoing related to $Q_k$, though \eqref{retarded jk} is proved as above for other cases. Namely, what we can prove by the Keel-Tao argument when $Q_j$ is not outgoing related to $Q_k$ is that $$\int\!\!\!\int_{t < s} \langle U_j(t) U_k^\ast(s) F(s), G(t) \rangle \, dsdt \leq C \|F\|_{L^2_t L_z^{r'}} \|G\|_{L^2_tL_z^{r'}}.$$ Nevertheless, noting that the homogeneous Strichartz estimates, by duality, implies $$\int\!\!\!\int \langle U_j(t) U_k^\ast(s) F(s), G(t) \rangle \, dsdt \leq C \|F\|_{L^2_t L_z^{r'}} \|G\|_{L^2_tL_z^{r'}}.$$ So we still obtain \eqref{retarded jk}.

\section{Mapping properties of Schr\"{o}dinger propagators}

It remains to prove Lemma \ref{long time lemma} and Lemma \ref{short time lemma}.

\begin{proof}[Proof of \eqref{dispersive low} and \eqref{dispersive high}]

The short time behaviour \eqref{dispersive low} and \eqref{dispersive high} come from the interpolation among \begin{eqnarray}
\label{short time 1 to r}\|\cdot\|_{L^1 \rightarrow L^r} &\leq& C t^{-(n + 1)/2} \quad \mbox{for any $r > 2$,}\\
\label{short time r' to infty}\|\cdot\|_{L^{r'} \rightarrow L^\infty} &\leq& C t^{-(n + 1)/2} \quad \mbox{for any $r > 2$,}\\
\label{short time 2 to 2}\|\cdot\|_{L^{2} \rightarrow L^2} &\leq& C  . 
\end{eqnarray}

The last one \eqref{short time 2 to 2} is indeed Proposition \ref{energy estimate}, while we shall prove  \eqref{short time 1 to r} and \eqref{short time r' to infty}, via dispersive estimates, by a comparison argument with hyperbolic space. The identical argument is used to show the restriction theorem in \cite{Chen-Hassell2}. 

Observing the RHS of the short time dispersive estimates in Proposition \ref{dispersive estimate1},  let us consider the kernel $$K_t(z, z') =  t^{-(n + 1)/2} (1 + d(z, z'))^n e^{- n d(z, z') / 2}$$ instead of the propagators. We claim, for $r > 2$, $$\|K_t\|_{L^1 \rightarrow L^r} = \sup_{z'} \|K_t\|_{L^r_z} \leq t^{-(n + 1)/2}.$$ Thinking of $K_t$ as a function supported on $X_0^2$, we decompose it as $$K_t = K_t \cdot \chi_U + K_t \cdot (1 - \chi_U),$$ where $U$ is a small neighbourhood of the front face.

The second part is proved by the fact that $d(z, z')$ is comparable to $- \log(xx')$ away from the front face. Then we have \begin{eqnarray*}\sup_{z'} \|K_t\|_{L^r_z} &=& t^{-(n + 1)/2}\sup_{z'}\bigg(\int (1 + d(z, z'))^{nr/2} e^{-nrd(z, z')/2} dg_z\bigg)^{1/r} \\ &\leq& C t^{-(n + 1)/2}\sup_{x'} \int \big(- \log (xx')\big)^{nr/2} (xx')^{nr/2} \, \frac{dx}{x^{n}} \\ &\leq& C t^{-(n + 1)/2}.\end{eqnarray*} 

On the other hand, consider the spectral measure restricted to $U$ say $K_{t, U} (z, z') = K_t \cdot \chi_U$. Before looking into the specific estimate, we shall compare this region with hyperbolic space $\mathbb{H}^{n + 1}$. To do so, one may further decompose the set $U$ into subsets $U_i$, where on each $U_i$, we have $x \leq \epsilon, x' \leq \epsilon$ and $d(y, y_i), d(y', y_i) \leq \epsilon$ for some $y_i \in \partial X$ (where the distance is measured with respect to the metric $h(y, dy)$ on $\partial X$). Choose local coordinates $(x, y)$ on $X$, centred at $(0, y_i) \in \partial X$, covering the set $ V_i = \{ x \leq \epsilon, d(y, y_i) \leq \epsilon \}$, and use these local coordinates to define a map $\phi_i$ from $V_i$ to a neighbourhood $V_i'$ of $(0,0)$ in hyperbolic space $\HH^{n+1}$ using the upper half-space model (such that the map is the identity in the given coordinates). The map $\phi_i$ induces a diffeomorphism $\Phi_i$ from $U_i \subset X^2_0$ to a subset of $(\mathbb{H}^{n+1})^2_0$, the double space for $\mathbb{H}^{n+1}$, covering the set $x \leq \epsilon, x' \leq \epsilon, |y|, |y'| \leq \epsilon$ in this space. Clearly, this map identifies $\rho_L$ and $\rho_R$ on $U_i$ with corresponding boundary defining functions for the left face and right face on $(\mathbb{H}^{n+1})^2_0$. We now reduce the kernel to
\begin{equation}
\phi_i \circ K_{t, U_i} \circ \phi_i^{-1}
\label{transfer}\end{equation} as an integral operator on $(\mathbb{H}^{n+1})^2_0$. After linking the front face to the hyperbolic case, we now can reduce to the estimate to the hyperbolic case as follows. $$\sup_{z'} \|K\|_{L^r_z(V_i)} = C \sup_{z'} \|\tilde{K}\|_{L^r_z(V_i')} \leq C |t|^{-(n + 1)/2},$$ where $K$ is mapped to $\tilde{K}$ on hyperbolic space and the $L^r$ norm of $\tilde{K}$ on hyperbolic space.

\end{proof}

\begin{proof}[Proof of \eqref{dispersive low long} and \eqref{dispersive high long}]

The long time behaviour results from the interpolation among \begin{eqnarray}
\label{long time 1 to r}\|\cdot\|_{L^1 \rightarrow L^r} \leq C t^{-3/2} && \mbox{for any $r > 2$,}\\
\label{long time r' to infty}\|\cdot\|_{L^{r'} \rightarrow L^\infty} \leq C t^{-3/2} && \mbox{for any $r > 2$,}\\
\label{long time r' to r}\|\cdot\|_{L^{r'} \rightarrow L^r} \leq C t^{-3/2} && \mbox{for any $r > 2$,}
\end{eqnarray} provided $t > 1$.

The proofs of \eqref{long time 1 to r} and \eqref{long time r' to infty} are exactly the same with \eqref{short time 1 to r} and \eqref{short time r' to infty}.

The novelty in the proof of \eqref{long time r' to r} is a non-trivial non-Euclidean ingredient called the Kunze-Stein phenomenon, which is named after Kunze and Stein \cite{Kunze-Stein}. Specifically, the Kunze-Stein phenomenon on hyperbolic space $\mathbb{H}^{n + 1}$ at $(2, 2)$  is expressed as \begin{equation*}\| f \ast F \|_{L^2(\mathbb{H}^{n + 1})} \leq C \|f\|_{L^2(\mathbb{H}^{n + 1})} \cdot \int_0^\infty |F(\rho)| (1 + \rho) e^{n\rho/2} d\rho,\end{equation*} for any $f, F \in C_0(\mathbb{H}^{n + 1})$, provided $F(\rho)$ is a radial function. See Cowling's work \cite{cowling-annmath} for a general result on semi-simple Lie groups. There is a generalized inequality \begin{equation}\label{Kunze-Stein2}\| f \ast F \|_{L^r(\mathbb{H}^{n + 1})} \leq C \|f\|_{L^{r'}(\mathbb{H}^{n + 1})} \cdot \bigg(\int_0^\infty |F(\rho)|^{r/2} (1 + \rho) e^{n\rho/2} d\rho\bigg)^{2/r}, \quad \mbox{for $r \geq 2$}\end{equation} obtained by Anker and Pierfelice \cite{Anker-Pierfelice2}. 

According to the long time dispersive estimates \eqref{long time dispersive microlocalized} \eqref{short time low energy dispersive},\footnote{Note $t^{-(n + 1)/2} < t^{-3/2}$ for the intermediate times $1 < |t - s| < 1 + d(z, z')$.}  we consider a  kernel $K_t(z, z') = t^{-3/2} (1 + d(z, z'))
e^{-n d(z, z')/2}$ on $X^2_0$ and decompose it as as $$K_t = K_t \cdot \chi_U + K_t \cdot (1 - \chi_U),$$ where $U$ is a small neighbourhood of the front face. 

The part away the front face is proved like the short time case $$ \|K_t\|_{L^{r'} \rightarrow L^r} = \|K_t\|_{L^r(X^2_0 \setminus U)} \leq t^{-3/2}\bigg(\int (1 + d(z, z'))^{nr/2} e^{-nrd(z, z')} dg_zdg_{z'}\bigg)^{1/r} \leq C t^{-3/2}.$$

For the part near the front face, we link the front face to hyperbolic space as in \eqref{transfer} and then have $$\bigg\|\int K \ast f\bigg\|_{L^{r}(V_i)} =  C \bigg\|\int \tilde{K} \ast \tilde{f}\bigg\|_{L^{r}(V_i')}  \leq C t^{-3/2} \| f\|_{L^{r'}(V_i)},$$ by invoking \eqref{Kunze-Stein2}, where $K$ and $f$ are mapped to $\tilde{K}$ and $\tilde{f}$ on hyperbolic space respectively.

\end{proof}

\section*{Acknowledgement}

This work was initiated in the research proposal of the application for the trimester program "Harmonic Analysis and Partial Differential Equations" hosted by Hausdorff Research Institute for Mathematics. Part of this paper was finished during my stay in Bonn. I would like to thank Hausdorff Research Institute for Mathematics for their hospitality. I am indebted to my PhD supervisor Andrew Hassell for his inexhaustible patience with me, without which this series of papers on asymptotically hyperbolic manifolds wouldn't have been finished. I am also grateful to Margaux Schmeltz for helping me write a French abstract.

\begin{flushleft}

\textsc{Shanghai Center for Mathematical Sciences\\Fudan University\\Shanghai 200433, China}

and

\textsc{Mathematical Sciences
Institute\\Australian National University\\Canberra 0200, Australia}

and

\textsc{Riemann Center for Geometry and Physics\\Leibniz Universit\"{a}t Hannover\\Hannover 30167, Germany}

\emph{E-mail address}: \textsf{herr.chenxi@outlook.com}

\end{flushleft}


\begin{thebibliography}{99}

\bibitem{Anker-Pierfelice}J.-P. Anker and V. Pierfelice, \emph{Nonlinear Schr\"{o}dinger equation on real hyperbolic spaces}, Ann. I. H. Poincar\'{e} AN \textbf{26}(2009), 1853-1869.

\bibitem{Anker-Pierfelice2}J.-P. Anker and V. Pierfelice, \emph{Wave and Klein-Gordon equations on hyperbolic spaces}, Anal. PDE \textbf{7}(2014), 953-995.

\bibitem{Banica-Carles-Staffilani}V. Banica, R. Carles and G. Staffilani, \emph{Scattering theory for radial nonlinear Schr\"{o}dinger equations on hyperbolic spaces}, Geom. Funct. Anal. \textbf{18}(2008), No. \textbf{2}, 367-399.

\bibitem{Bouclet-apde-2011}J.-M. Bouclet, \emph{Strichartz estimates on asymptotically hyperbolic manifolds}, Anal. PDE \textbf{4}(2011), No. \textbf{1}, 1-84.

\bibitem{Bourgain-GAFA-1993-1}J. Bourgain, \emph{Fourier transform restriction phenomena for certain lattice subsets and applications to nonlinear evolution equations: I. Schr\"{o}dinger equations}, Geom. Funct. Anal. \textbf{3}(1993), 107-156.

\bibitem{Bourgain-GAFA-1993-2}J. Bourgain, \emph{Exponential sums and nonlinear Schr\"{o}dinger equations}, Geom. Funct. Anal. \textbf{3}(1993), 157-178.


\bibitem{Burq-Gerard-Tzvetkov-2004} N. Burq, P. G\'{e}rard and N. Tzvetkov, \emph{Strichartz inequalities and the nonlinear Schr\"{o}dinger equation on compact manifolds}, Amer. J. Math. \textbf{126}(2004), 569-605.

\bibitem{Burq-Gerard-Tzvetkov-2005} N. Burq, P. G\'{e}rard and N. Tzvetkov, \emph{Bilinear eigenfunction estimates and the nonlinear Schr\"{o}dinger equation on surfaces}, Invent. Math. \textbf{159}(2005), 187-223.


\bibitem{Burq-Guillarmou-Hassell}N. Burq, C. Guillarmou and A. Hassell, \emph{Strichartz estimates without loss on manifolds with hyperbolic trapped geodesics}, Geom. Funct. Anal. \textbf{20}(2010), 627-656.




\bibitem{Cazenave book}T. Cazenave, \emph{Semilinear Schr\"{o}dinger Equations}, Courant Lecture Notes in Mathematics, \textbf{10}, New York Univ., Courant Inst. in Math. Sci., New York; Amer. Math. Soc., Providence (2003).




\bibitem{Chen-Hassell1} X. Chen and A. Hassell, \emph{Resolvent and spectral measure on non-trapping asymptotically hyperbolic manifolds I: Resolvent construction at high energy}, arXiv:1410.6881.


\bibitem{Chen-Hassell2} X. Chen and A. Hassell, \emph{Resolvent and spectral measure on non-trapping asymptotically hyperbolic manifolds II: Spectral Measure, Restriction Theorem, Spectral Multiplier},  arXiv:1412.4427.



\bibitem{Christ-Kiselev}M. Christ and A. Kiselev, \emph{Maximal functions associated to filtrations,} J. Funct. Anal. \textbf{179}(2001), 409-425.


\bibitem{cowling-annmath}M. Cowling, \emph{The Kunze-Stein phenomenon}, Ann. Math. \textbf{107}(1978), 209-234.

\bibitem{Ginibre-Velo-JMPA-1985}J. Ginibre and G. Velo,  \emph{Scattering theory in the energy space for a class of nonlinear Schr\"{o}dinger equations}, J. Math. Pures Appl. \textbf{64}(1985), 363-401.

\bibitem{Ginibre-Velo-JFA-1995}J. Ginibre and G. Velo, \emph{Generalized Strichartz inequalities for the wave equation}, J. Funct. Anal. \textbf{133}(1995), No. \textbf{1}, 50-68.


\bibitem{fourier analysis}L. Grafakos, \emph{Classical and Modern Fourier Analysis}, Prentice Hall, New Jersey, 2004.



\bibitem{Guillarmou-Hassell}C. Guillarmou and A. Hassell, \emph{Uniform Sobolev estimates for non-trapping metrics},  J. Inst. Math. Jussieu \textbf{13}(2014), No.\textbf{3}, 599-632.

\bibitem{Guillarmou-Hassell-Sikora}C. Guillarmou, A. Hassell and A. Sikora, \emph{Restriction and spectral multiplier theorems on asymptotically conic manifolds}, Anal. PDE \textbf{6}(2013), No.\textbf{4}, 893-950.


\bibitem{Guillarmou-Qing}C. Guillarmou and J. Qing, \emph{Spectral characterization of Poincar\'{e}-Einstein manifolds with infinity of positive Yamabe type}, Inter. Math. Res. Not. \textbf{2010}, No. \textbf{9}, 1720-1740.


\bibitem{Hassell-Zhang}A. Hassell and J. Zhang, \emph{Global-in-time Strichartz estimates on non-trapping asymptotically conic manifolds}, arXiv:1310.0909.



\bibitem{Ionescu-Annmath-2000}A. Ionescu, \emph{An endpoint estimate for the Kunze-Stein phenomenon and related maximal operators,} Ann. Math.(2) \textbf{152}(2000), 259-275.

\bibitem{Ionescu-Staffilani-Mathann-2009}A. Ionescu and G. Staffilani, \emph{Semilinear Schr\"{o}dinger flows on hyperbolic spaces: scattering in $H^1$}, Math. Ann. \textbf{345}(2009), 133-158.


\bibitem{Kato}T. Kato, \emph{On nonlinear Schr\"{o}dinger equations}, Ann. Inst. H. Poincar\'{e} Phys. Th\'{e}or. \textbf{46}(1987), 113-129. 


\bibitem{Keel-Tao}M. Keel and T. Tao, \emph{Endpoint Strichartz estimates}, Amer. J. Math. \textbf{120}(1998), No.\textbf{5}, 955-980.


\bibitem{Kunze-Stein}R. A. Kunze and E. M. Stein, \emph{Uniformly bounded representations and harmonic analysis of the $2 \times 2$ unimodular group}, Amer. J. Math. \textbf{82}(1960), 1-62.




\bibitem{Mazzeo-JDG-1988}R. Mazzeo, \emph{The Hodge cohomology of a conformally compact metric},  J. Diff. Geom. \textbf{28}(1988), 309-339.


\bibitem{Mazzeo-1991}R. Mazzeo, \emph{Unique continuation at infinity and embedded eigenvalues for asymptotically hyperbolic manifolds},  Amer. J. Math. \textbf{113} (1991), no. 1, 25-45. 

\bibitem{Mazzeo-Melrose}R. Mazzeo and R. B. Melrose, \emph{Meromorphic extention of the resolvent on complete spaces with asymptotically constant negative curvature}, J. Func. Anal. \textbf{75}(1987), 260-310.

\bibitem{Melrose-Sa Barreto-Vasy}R. B. Melrose, A. S\'{a} Barreto and A. Vasy, \emph{Analytic continuation and semiclassical resolvent estimates on asymptotically hyperbolic spaces,} Comm. Part. Diff. Equa. \textbf{39}(2014), 452-511.



\bibitem{Strichartz}R. S. Strichartz, \emph{Restrictions of Fourier transforms to quadratic surfaces and decay of solutions of wave equations,} Duke Math. J. \textbf{44}(1977), No. \textbf{3}, 705-714.

\bibitem{Tao book}T. Tao, \emph{Nonlinear Dispersive Equations, Local and Global Analysis.} CBMS Regional Conference Series in Math. \textbf{106},  Amer. Math. Soc., Providence, RI, 2006.


\bibitem{Wang}Y. Wang, \emph{Resolvent and radiation fields on asymptotically hyperbolic manifolds}, arXiv:1410.6936.







\end{thebibliography}
\end{document}